\documentclass{amsart}
\pdfoutput=1
\usepackage[margin=3.1cm]{geometry}
\usepackage{graphicx,amsmath, amsthm, amssymb, calrsfs, wasysym, verbatim, bbm, color, graphics, parskip, setspace, bbm, quiver, hyperref, eurosym, booktabs, multicol}
\usepackage[mathscr]{eucal}
\usepackage{tikz-cd}
\tikzcdset{scale cd/.style={every label/.append style={scale=#1},
    cells={nodes={scale=#1}}}}
\usepackage[
backend=biber,
style=alphabetic, maxbibnames=99, maxcitenames=99]{biblatex}
\hypersetup{
    colorlinks=true,
    linkcolor=black,
    citecolor= brown,
      }
\usepackage[T1]{fontenc}
\DeclareMathAlphabet{\mathcal}{OMS}{cmsy}{m}{n}

\theoremstyle{remark}
\newtheorem{defn}{Definition}[section]

\newtheorem{assumption}[defn]{Assumption}

\newtheorem{remark}[defn]{Remark}

\theoremstyle{definition}
\newtheorem{prop}[defn]{Proposition}
\newtheorem{lemma}[defn]{Lemma}
\newtheorem{cor}[defn]{Corollary}
\newtheorem{thm}[defn]{Theorem}

\newcommand{\cat}[1]{\mathscr{#1}}
\newcommand{\Spc}{\operatorname{Spc}}
\newcommand{\Spec}{\operatorname{Spec}}
\newcommand{\pdom}{\stackrel{p}{\geq}}
\newcommand{\twodom}{\stackrel{2}{\geq}}
\newcommand{\peq}{\stackrel{p}{\sim}}

\newcommand{\supp}{\operatorname{supp}}
\newcommand{\unit}{\mathbbm{1}}
\newcommand{\tunit}{\tilde{\unit}}
\newcommand{\colim}{\mathrm{colim}}
\newcommand{\Hom}{\mathrm{Hom}}

\DeclareMathOperator{\geom}{gm}
\DeclareMathOperator{\DM}{DM}
\DeclareMathOperator{\DMgm}{\DM_{\geom}}
\DeclareMathOperator{\DQM}{DQM}
\DeclareMathOperator{\DTM}{DTM}
\DeclareMathOperator{\DTMgm}{\DTM_{\geom}}
\DeclareMathOperator{\DAM}{DAM}
\DeclareMathOperator{\DAMgm}{\DAM_{\geom}}
\DeclareMathOperator{\DQMgm}{\DQM_{\geom}}
\DeclareMathOperator{\DATM}{DATM}
\DeclareMathOperator{\DATMgm}{\DATM_{\geom}}
\DeclareMathOperator{\PP}{\mathbb{P}}
\DeclareMathOperator{\F}{\mathbb{F}}
\DeclareMathOperator{\N}{\mathbb{N}}
\DeclareMathOperator{\Q}{\mathbb{Q}}
\DeclareMathOperator{\C}{\mathbb{C}}
\DeclareMathOperator{\R}{\mathbb{R}}
\DeclareMathOperator{\Z}{\mathbb{Z}}
\DeclareMathOperator{\M}{M}
\DeclareMathOperator{\T}{T}
\DeclareMathOperator{\cont}{\mathfrak{c}}
\DeclareMathOperator{\p}{\mathfrak{p}}

\DeclareMathOperator{\hocolim}{hocolim}
\DeclareMathOperator{\Loc}{Loc}
\DeclareMathOperator{\K}{K}
\DeclareMathOperator{\cone}{cone}
\DeclareMathOperator{\im}{im}
\begin{document}
\title{The motivic tt-geometry of real quadrics}
\author{Jean Paul Schemeil}
\address{School of Mathematical Sciences, University of Nottingham, Nottingham, NG7 2RD, UK}
\email{jeanpaul.schemeil@nottingham.ac.uk}
\date{\today}
\begin{abstract}
We study the tensor-triangular geometry of the category of Voevodsky motives generated by real quadrics. At the prime 2, we determine its Balmer spectrum, and find that it is a countably infinite, non-Noetherian space of Krull dimension 2. We detail the relationship between this space, the real Artin-Tate spectrum computed by Balmer–Gallauer, and Vishik’s isotropic points. We conclude by combining our computation with Balmer–Gallauer's results on Artin-Tate motives to obtain a full description of the spectrum of integral motives of quadrics over real algebraic numbers.
\end{abstract}
\maketitle
\tableofcontents
\section{Introduction}
Tensor-triangular geometry (\emph{tt-geometry} for short) has found fruitful applications in many settings \cite{balmer2010tensor}. A particularly active direction is \emph{motivic tt-geometry}, whose long-term goal is to classify motives up to the tensor-triangulated structure of their ambient category
\[
\DM(k):=\DM(k;\Z),
\]
introduced by Voevodsky \cite{voevodsky2000triangulated}. For \emph{geometric motives} (the compact objects of $\DM(k)$), this amounts to describing the Balmer spectrum \cite{balmer2005spectrum} of the tt-category $\DMgm(k)$. Although the current status of the problem remains largely open, recent work of Vishik \cite{vishik2024isotropic, vishik2025balmer} has provided profound insights into the complexity of this space. When the characteristic of the base field $k$ is zero, Vishik defined for any prime number $p$ the category of \emph{isotropic motives} $\DM(k/k; \F_p)$ by localising $\DM(k; \F_p)$ at motives of $p$-anisotropic varieties (those whose closed points have degree divisible by $p$). Using this framework, Vishik constructed a family of \emph{isotropic realisation} functors
\[\psi_{E,p}: \DM(k) \to \DM(E(\PP^{\infty})/E(\PP^{\infty}); \F_p)\]
indexed over all primes $p$ and field extensions $E/k$. When restricted to geometric motives, the kernels of these 
functors detect the largest currently known family of closed points of the spectrum of $\DMgm(k)$. As a striking consequence, over the real numbers $k=\R$, these realisations allowed Vishik to establish that $\Spc(\DMgm(\R))$ has cardinality $2^{\mathfrak{c}}$. This bound provides a first measure of the complexity of this space, anticipating the challenge involved in achieving a complete topological classification.

A sensible approach to navigate this complexity is to study the spectrum of the motivic category "from below": rather than confronting $\DMgm(k)$ all at once, one considers the tt-geometry of carefully chosen full tt‑subcategories of the motivic category. The relevance of this strategy relies on a general theorem of Balmer \cite{balmer2018surjectivity} who showed that the inclusion of a full tt-subcategory 
\[\cat{L} \hookrightarrow \cat{K}\] induces a continuous surjective restriction map on spectra 
\[\Spc(\cat{K}) \twoheadrightarrow \Spc(\cat{L}).\] 
This surjection effectively provides a lower bound on the complexity of the ambient spectrum. In the motivic setting, manifestations of this idea are exemplified by the work of Peter \cite{peter2013prime} and Gallauer \cite{gallauer2019tt} who initiated the study of the spectrum of \emph{Tate motives} $\DTMgm(k)$; Balmer-Gallauer \cite{balmer2022three, balmer2025spectrum} who introduced the study of the spectrum of \emph{Artin motives} $\DAMgm(k)$ and \emph{Artin-Tate motives} $\DATMgm(k)$; and Sparks \cite{sparks2025tensor} who studied the tt-geometry of these categories in isotropic contexts. The relevance of restricting to such full tt-subcategories is that their tt-geometry can often be described explicitly in terms of specific algebraic invariants of the base field. For example, over algebraically closed and real closed fields, Gallauer \cite{gallauer2019tt} and Balmer-Gallauer \cite{balmer2022three} demonstrated that the spectrum of Tate motives admits a description in terms of the \textit{motivic cohomology} of the base field; for Artin motives, the recent breakthrough of Balmer-Gallauer \cite{balmer2025spectrum} has shown that, over any field, the spectrum can be expressed in terms of the \textit{absolute Galois group} of the base field. Furthermore, Artin-Tate motives over real closed fields, Balmer-Gallauer \cite{balmer2022three} have obtained a corresponding description involving a combination of the above invariants. Finally, in isotropic contexts, Sparks \cite{sparks2025tensor} has shown that at the prime $2$, the tt-geometry of \emph{isotropic (Artin-)Tate motives over flexible fields} can be described in terms of the \emph{isotropic motivic cohomology} of the base field.

Following the themes outlined above, in this article we propose to initiate the study of the Balmer spectrum of the full tt-subcategory generated by \emph{motives of smooth projective quadrics} 
\[\DQMgm(k; R).\] 
Originally introduced by Bachmann \cite{bachmann2017invertibility} to study the\emph{ Picard group }of the motivic category, this subcategory provides a natural first candidate to incorporate motives of higher dimensional varieties into the landscape of motivic tt-geometry. To set the stage, we briefly address the case when the base field $k$ is algebraically closed. Because all quadratic forms over such a field are isotropic, the motives of quadrics completely decompose into Tate motives; consequently, $\DQMgm(k, R)=\DTMgm(k, R)$. Thanks to the work of Peter \cite{peter2013prime} and Gallauer \cite{gallauer2019tt}, the tt-geometry of this category is already well understood. However, if $k$ is not algebraically closed, $\DTMgm(k ; R)$  forms a proper tt-subcategory of $\DQMgm(k ; R)$. In general, it is expected that the tt-geometry of $\DQMgm(k ; R)$ reflects the complexity of the theory of quadratic forms over $k$: this is because $\DQMgm(k ; R)$ is generated by Tate motives and by \emph{motives of anisotropic quadrics}, which correspond to anisotropic quadratic forms over $k$. This naturally identifies the case $k=\R$ (or more generally, any real closed field) as the natural first step beyond the algebraically closed fields, and it will represent the primary focus of this article. Indeed, we have very good control over the theory of quadratic forms over $\R$; moreover, if $R=\F_p$ for $p$ an odd prime (or more generally, $2$ is invertible in $R$), then motives of quadrics decompose into Tate motives and motives of quadratic extensions (i.e. $\M(\C)$), in particular we have an equivalence
\[ \DATMgm(\R ;\F_p) \simeq  \DQMgm(\R ; \F_p) . \]
By \cite[Corollary 11.2]{balmer2022three}, this spectrum is completely understood. However, for $R=\F_2$ (or more generally, $2$ is not a unit of $R$), the inclusion 
\[\DATMgm(\R ; \F_2) \hookrightarrow \DQMgm(\R ; \F_2)\]
is far from being an equivalence, since motives of arbitrary dimensional $2$-anisotropic quadrics enter the picture, making the tt-category $\DQMgm(\R ; \F_2)$ substantially richer. The main result of this article is the computation of the Balmer spectrum of this tt-category.

\textbf{Main Theorem}. \emph{Let $F$ be a real closed field. The Balmer spectrum of the tt-category $\DQMgm(F ; \F_2)$ is the following countably infinite, non-Noetherian space of Krull dimension 2, }
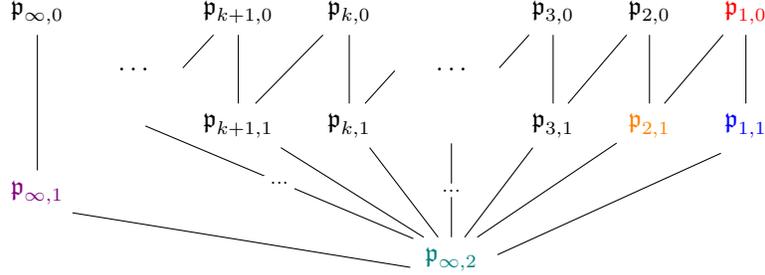
\[\begin{tikzcd}[ampersand replacement=\&,sep=small]
	{\p_{\infty,0}} \&\& {\p_{k+1,0}} \& {\p_{k,0}} \& {} \& {\p_{3,0}} \& {\p_{2,0}} \& {\textcolor{red}{\p_{1,0}}} \\
	\& \dots \&\&\& \dots \\
	\& {} \& {\p_{k+1,1}} \& {\p_{k,1}} \& {} \& {\p_{3,1}} \& {\textcolor{orange}{\p_{2,1}}} \& {\textcolor{blue}{\p_{1,1}}} \\
	{\textcolor{violet}{\p_{\infty,1}}} \\
	\&\&\&\& {\textcolor{teal}{\p_{\infty,2}}}
	\arrow[no head, from=1-1, to=4-1]
	\arrow[between={0}{0.4}, no head, from=1-3, to=3-2]
	\arrow[no head, from=1-3, to=3-3]
	\arrow[no head, from=1-4, to=3-3]
	\arrow[no head, from=1-4, to=3-4]
	\arrow[between={0}{0.4}, no head, from=1-6, to=3-5]
	\arrow[no head, from=1-6, to=3-6]
	\arrow[no head, from=1-7, to=3-6]
	\arrow[no head, from=1-7, to=3-7]
	\arrow[no head, from=1-8, to=3-7]
	\arrow[no head, from=1-8, to=3-8]
	\arrow["\dots"{description}, no head, from=3-2, to=5-5]
	\arrow[no head, from=3-3, to=5-5]
	\arrow[between={0}{0.4}, no head, from=3-4, to=1-5]
	\arrow[no head, from=3-4, to=5-5]
	\arrow[no head, from=3-6, to=5-5]
	\arrow[no head, from=3-7, to=5-5]
	\arrow[shift right=3, no head, from=4-1, to=5-5]
	\arrow["\dots"{description, pos=0.3}, between={0}{0.6}, no head, from=5-5, to=2-5]
	\arrow[shift right=3, no head, from=5-5, to=3-8]
\end{tikzcd}\]
\emph{where a line indicates the specialisation relation going upwards. The closed subsets are specialisation closed subsets $Z$ such that if $\p_{\infty,0} \not \in Z$, then $Z$ is finite.}

The colours decorating the picture above depict how points get mapped under the restriction
\[ \Spc(\DQMgm(F ; \F_2)) \twoheadrightarrow \Spc(\DATMgm(F ; \F_2)),\]
where the (decorated) image consists of the following six-point space \cite[Theorem 10.1]{balmer2022three}.
\[\begin{tikzcd}[ampersand replacement=\&,sep=small]
	{\mathcal{L}_1} \&\&\&\& {\textcolor{red}{\mathcal{N}_1}} \\
	\&\& {\textcolor{orange}{\mathcal{M}_1}} \\
	{\textcolor{violet}{\mathcal{L}_0}} \&\&\&\& {\textcolor{blue}{\mathcal{N}_0}} \\
	\&\& {\textcolor{teal}{\mathcal{M}_0}}
	\arrow[no head, from=1-1, to=2-3]
	\arrow[no head, from=1-1, to=3-1]
	\arrow[no head, from=1-5, to=3-5]
	\arrow[no head, from=2-3, to=1-5]
	\arrow[no head, from=2-3, to=4-3]
	\arrow[no head, from=3-1, to=4-3]
	\arrow[no head, from=4-3, to=3-5]
\end{tikzcd}\]
From the description of the topology, we observe that the top horizontal sequence of closed points $\p_{k,0}$ for $k \in \N$, as well as the lower horizontal sequence of points $\p_{k,1}$ for $k \in \N$, both converge to the point at infinity $\p_{\infty,0}$.
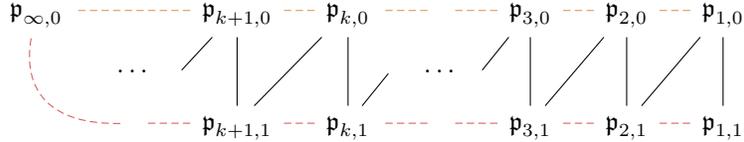
\[\begin{tikzcd}[ampersand replacement=\&,sep=small]
	{\p_{\infty,0}} \&\& {\p_{k+1,0}} \& {\p_{k,0}} \& {} \& {\p_{3,0}} \& {\p_{2,0}} \& {\p_{1,0}} \\
	\& \dots \&\&\& \dots \\
	\& {} \& {\p_{k+1,1}} \& {\p_{k,1}} \& {} \& {\p_{3,1}} \& {\p_{2,1}} \& {\p_{1,1}}
	\arrow[color={rgb,255:red,214;green,153;blue,92}, dashed, no head, from=1-3, to=1-1]
	\arrow[between={0}{0.4}, no head, from=1-3, to=3-2]
	\arrow[no head, from=1-3, to=3-3]
	\arrow[color={rgb,255:red,214;green,153;blue,92}, dashed, no head, from=1-4, to=1-3]
	\arrow[no head, from=1-4, to=3-3]
	\arrow[no head, from=1-4, to=3-4]
	\arrow[color={rgb,255:red,214;green,153;blue,92}, dashed, no head, from=1-5, to=1-4]
	\arrow[color={rgb,255:red,214;green,153;blue,92}, dashed, no head, from=1-6, to=1-5]
	\arrow[between={0}{0.4}, no head, from=1-6, to=3-5]
	\arrow[no head, from=1-6, to=3-6]
	\arrow[color={rgb,255:red,214;green,153;blue,92}, dashed, no head, from=1-7, to=1-6]
	\arrow[no head, from=1-7, to=3-6]
	\arrow[no head, from=1-7, to=3-7]
	\arrow[color={rgb,255:red,214;green,153;blue,92}, dashed, no head, from=1-8, to=1-7]
	\arrow[no head, from=1-8, to=3-7]
	\arrow[no head, from=1-8, to=3-8]
	\arrow[color={rgb,255:red,214;green,92;blue,92}, curve={height=-24pt}, dashed, no head, from=3-2, to=1-1]
	\arrow[color={rgb,255:red,214;green,92;blue,92}, dashed, no head, from=3-3, to=3-2]
	\arrow[between={0}{0.4}, no head, from=3-4, to=1-5]
	\arrow[color={rgb,255:red,214;green,92;blue,92}, dashed, no head, from=3-4, to=3-3]
	\arrow[color={rgb,255:red,214;green,92;blue,92}, dashed, no head, from=3-5, to=3-4]
	\arrow[color={rgb,255:red,214;green,92;blue,92}, dashed, no head, from=3-6, to=3-5]
	\arrow[color={rgb,255:red,214;green,92;blue,92}, dashed, no head, from=3-7, to=3-6]
	\arrow[color={rgb,255:red,214;green,92;blue,92}, dashed, no head, from=3-8, to=3-7]
\end{tikzcd}\]
This behaviour justifies the asymmetrical layout we have chosen to draw $\Spc(\DQMgm(F ; \F_2))$.

The computation of $\Spc(\DQMgm(F ; \F_2))$ uses the following three main ideas. 
\subsubsection*{Reduction to cellular motivic categories}
We construct a family of geometric functors out of the big tt-category $\DQM(F ; \F_2)$ into more tractable cellular motivic tt-categories
\[\{\pi_n\colon \DQM(F ;\F_2)\to \cat{K}_n\}_{n\in \N\cup\{\infty\}}.\]
The categories $\cat{K}_n$ are obtained by base-changing along the function fields of particular anisotropic quadrics over $F$, and then applying a smashing localisation which kills the motives of those quadrics that remain anisotropic. This stratification technique goes back to \cite{bachmann2017invertibility, vishik2019affine}, where it was used to study the Picard group of $\DQMgm(k ; R)$. Transposing it to tt-geometry, we show that the $\pi_n$ form a jointly conservative family; invoking the general theorem of \cite{barthel2024surjectivity}, this reduces the computation of $\Spc(\DQMgm(F ; \F_2))$ to the spectra $\Spc(\cat{K}_n^c)$, where $\cat{K}_n^c$ denotes the full tt-subcategory of the compact objects of $\cat{K}_n$.

\subsubsection*{Computing $\Spc(\cat{K}_n^c)$ using weight structures.}
 We show that the tt-geometry of the cellular categories $\cat{K}_n^c$ is determined by the \emph{bi-graded} endomorphism ring of the unit. The key input is the construction of suitable weight structures, which reduce the computation of spectra to proving that sufficiently many morphisms are tensor-nilpotent on specified objects. This technique is inspired by Sparks \cite{sparks2025tensor}, who proved that under appropriate hypotheses the existence of enough tensor-nilpotent morphisms forces the spectrum of a tt-category to be a singleton. In our setting, we adapt this idea to control tensor-nilpotence on specific objects, providing an effective computational tool.

\subsubsection*{Chow weight structure and generators of prime tt-ideals.}
After detecting all points of the Balmer spectrum of $\DQMgm(F ; \F_2)$, it remains to determine specialisation relations and the topology. We do this by describing generators for all tt-primes. The key tool is the classical Chow weight structure restricted to $\DQMgm(F ; \F_2)$, which will allow us to study the pull-backs of points of $\Spc(\cat{K}_n^c)$ along the functors $\pi_n$.

After establishing the main theorem and its relationship with the spectrum of real Artin-Tate motives, we analyse how the $2^{\cont}$ isotropic points detected by Vishik restrict to $\DQMgm(\R ; \F_2)$. Finally, since over real closed fields motives of quadrics coincide with Artin-Tate motives away from the prime $2$, we explain how to obtain the full integral spectrum when the algebraic closure of $F$ satisfies the \emph{Peter-Gallauer vanishing condition} (cf. \cite[Hypothesis 6.6]{gallauer2019tt}) and \cite[Section 4.3]{peter2013prime})--for instance, when $F$ is the field of real algebraic numbers $\R^{\mathrm{alg}}$--by adapting the strategy of \cite[Theorem 11.3]{balmer2022three}.

\textbf{Corollary}. \emph{Let $F=\R^{alg}=\overline{\Q} \cap \R$, then the Balmer spectrum of $\DQMgm(F ; \Z)$ is the following space,}
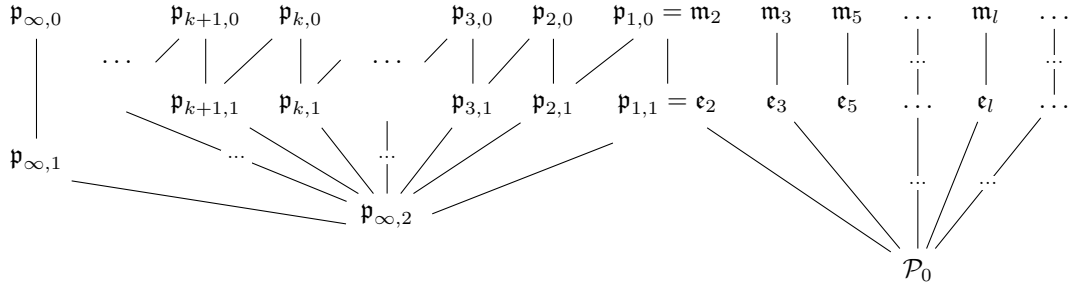
\[\begin{tikzcd}[ampersand replacement=\&,sep=tiny]
	{\p_{\infty,0}} \&\& {\p_{k+1,0}} \& {\p_{k,0}} \& {} \& {\p_{3,0}} \& {\p_{2,0}} \& {\p_{1,0}=\mathfrak{m}_2} \& {\mathfrak{m}_3} \& {\mathfrak{m}_5} \& \dots \& {\mathfrak{m}_l} \& \dots \\
	\& \dots \&\&\& \dots \\
	\& {} \& {\p_{k+1,1}} \& {\p_{k,1}} \& {} \& {\p_{3,1}} \& {\p_{2,1}} \& {\p_{1,1}=\mathfrak{e}_2} \& {\mathfrak{e}_3} \& {\mathfrak{e}_5} \& \dots \& {\mathfrak{e}_l} \& \dots \\
	{\p_{\infty,1}} \\
	\&\&\&\& {\p_{\infty,2}} \\
	\&\&\&\&\&\&\&\&\&\& {\mathcal{P}_0}
	\arrow[no head, from=1-1, to=4-1]
	\arrow[between={0}{0.4}, no head, from=1-3, to=3-2]
	\arrow[no head, from=1-3, to=3-3]
	\arrow[no head, from=1-4, to=3-3]
	\arrow[no head, from=1-4, to=3-4]
	\arrow[between={0}{0.4}, no head, from=1-6, to=3-5]
	\arrow[no head, from=1-6, to=3-6]
	\arrow[no head, from=1-7, to=3-6]
	\arrow[no head, from=1-7, to=3-7]
	\arrow[no head, from=1-8, to=3-7]
	\arrow[no head, from=1-8, to=3-8]
	\arrow[no head, from=1-9, to=3-9]
	\arrow[no head, from=1-10, to=3-10]
	\arrow["\dots"{description}, no head, from=1-11, to=3-11]
	\arrow[no head, from=1-12, to=3-12]
	\arrow["\dots"{description}, no head, from=1-13, to=3-13]
	\arrow["\dots"{description}, shift right, no head, from=3-2, to=5-5]
	\arrow[no head, from=3-3, to=5-5]
	\arrow[between={0}{0.4}, no head, from=3-4, to=1-5]
	\arrow[no head, from=3-4, to=5-5]
	\arrow[no head, from=3-6, to=5-5]
	\arrow[no head, from=3-7, to=5-5]
	\arrow[no head, from=3-8, to=6-11]
	\arrow[no head, from=3-9, to=6-11]
	\arrow["\dots"{description}, no head, from=3-11, to=6-11]
	\arrow[no head, from=3-12, to=6-11]
	\arrow["\dots"{description}, no head, from=3-13, to=6-11]
	\arrow[shift right=3, no head, from=4-1, to=5-5]
	\arrow["\dots"{description, pos=0.3}, between={0}{0.6}, no head, from=5-5, to=2-5]
	\arrow[shift right=3, no head, from=5-5, to=3-8]
\end{tikzcd}\]
\emph{where $l$ ranges over all prime numbers, and lines indicate specialisation relations going upwards. The closed subsets are specialisation closed subsets $Z$ such that
\begin{itemize}
    \item if $\p_{\infty,0} \not \in Z$, then $Z \cap \overline{\{\p_{\infty, 2}\}}$ is finite;
    \item if $\mathcal{P}_{0} \not \in Z$, then $Z \cap \overline{\{\mathcal{P}_{0}\}}$ is finite.
\end{itemize}}

\subsection*{Acknowledgements} 
I am deeply grateful to my advisor Alexander Vishik for many invaluable discussions, for his constant support and guidance, and for his insightful comments on a preliminary version of this document. I would like to thank Fraser Sparks for many fruitful conversations and for generously sharing with me his results on the Balmer spectrum of isotropic Tate motives over flexible fields prior to their appearance in \cite{sparks2025tensor}. I would also like to thank Martin Gallauer for his helpful comments on a preliminary version of this document.

\section{Preliminaries}
In this section we briefly recall the tt-geometric and motivic background used throughout the paper, and we fix notation and terminology.
\subsection{tt-geometric recollections}
Throughout, let $\cat{K}$ be a \emph{rigidly-compactly generated big tt-category}, and let $\cat{K}^c$ denote its essentially small rigid tt-subcategory of compact objects. Following \cite{balmer2005spectrum}, the \emph{Balmer spectrum} of $\cat{K}^c$ is the topological space
\[
\Spc(\cat{K}^c),
\]
whose underlying set consists of prime tt-ideals of $\cat{K}^c$, endowed with the topology for which a basis of closed sets is given by the supports
\[
\supp(X) := \{\,\p \in \Spc(\cat{K}^c) \mid X \notin \p\,\},
\]
where $X$ ranges over all objects of $\cat{K}^c$.

In these settings, the assignment
\[
\cat{J}\longmapsto \bigcup_{X\in \cat{J}} \supp(X)
\]
induces an inclusion-preserving bijection between (radical) tt-ideals $\cat{J}\subseteq \cat{K}^c$ and Thomason subsets of $\Spc(\cat{K}^c)$.
\subsection*{Jointly conservative families of geometric functors}
Let $F: \cat{K}\to \cat{L}$ be a tt-functor between rigidly-compactly generated big tt-categories. We say that $F$ is \emph{geometric} if it preserves coproducts. Any geometric tt-functor preserves compact objects, hence restricts to a tt-functor
\[
F^c : \cat{K}^c \to \cat{L}^c.
\]
By the functoriality of the Balmer spectrum, $F^c$ induces a continuous spectral map
\[
\Spc(F^c)\colon \Spc(\cat{L}^c)\to \Spc(\cat{K}^c),
\]
given by the pull-back of prime tt-ideals.

A family of geometric functors $\{F_i\colon \cat{K}\to \cat{K}_i\}_{i\in I}$ is called \emph{jointly conservative} if $F_i(X)=0$ for all $i\in I$ implies $X=0$.

The relevance of this notion lies in the following fundamental result in tt-geometry.

\begin{thm}[\cite{barthel2024surjectivity}] \label{theorem: barthel2024surjectivity}
\emph{If a family of geometric functors as above is jointly conservative, then the induced map on spectra
\[ \bigsqcup_{i \in I} \Spc(F^c_i): \bigsqcup_{i \in I} \Spc(\cat{K}_i^c) \to \Spc(\cat{K}^c) \]
is surjective. }
\end{thm}
\subsection*{Smashing localisations}{(\cite{balmer2011generalized})}
A full triangulated subcategory $\cat{J}\subseteq \cat{K}$ is called a \emph{localising tensor-ideal} if it is closed under arbitrary coproducts and satisfies $\cat{K}\otimes \cat{J}\subseteq \cat{J}$. For a collection of objects $S\subseteq \cat{K}$, we write $\Loc_{\otimes}(S)$ for the smallest localising tensor-ideal containing $S$.
\begin{remark}{(\cite[Theorem 8.1]{greenlees2019balmer}, \cite{stevenson2013support})}
If $S$ is a set of compact objects, then the tt-ideal $\langle S\rangle\subseteq \cat{K}^c$ identifies as
\[
\langle S\rangle \;=\; \Loc_{\otimes}(S)\cap \cat{K}^c.
\]
\end{remark}
A \emph{localisation} of $\cat{K}$ is a pair $(L, \lambda)$, where $L: \cat{K} \to \cat{K}$ is an exact functor and $\lambda$ is a natural transformation $\lambda: \mathrm{Id}_{\cat{K}} \to L$ such that $L\lambda=\lambda L$ and $L \lambda: L \to L^2$ is an isomorphism. A colocalisation of $\cat{K}$ is a pair $(\Gamma, \gamma)$ such that $(\Gamma^{op}, \gamma^{op})$ is a localisation of $\cat{K}^{op}$.

A localisation $(L, \lambda)$ of $\cat{K}$ is called \emph{smashing} if $L$ preserves coproducts. A localising tensor-ideal $\cat{S}\subseteq \cat{K}$ is called \emph{smashing} if there exists a smashing localisation $(L_{\cat{S}}, \lambda)$ of $\cat{K}$ such that $\ker(L_{\cat{S}})=\cat{S}$. In this case, there is an associated colocalisation functor $(\Gamma_{\cat{S}}, \gamma)$ such that $\im(\Gamma_{\cat{S}})=\cat{S}$. Moreover, there is an idempotent triangle
\[
\Gamma_{\cat{S}}(\unit)\longrightarrow \unit \longrightarrow L_{\cat{S}}(\unit)\longrightarrow \Gamma_{\cat{S}}(\unit)[1].
\]
Crucially, in these settings, smashing localisations are in one-to-one correspondence with idempotent triangles \cite[Theorem 3.5]{balmer2011generalized}. In particular, there are natural isomorphisms \[
L_{\cat{S}} \simeq L_{\cat{S}}(\unit)\otimes -,
\qquad
\Gamma_{\cat{S}} \simeq \Gamma_{\cat{S}}(\unit)\otimes -.
\]
Conversely, any  idempotent triangle in $\cat{K}$ 
\[X \to  \unit \to Y \to X[1] \]
gives rise to a smashing localisation $L_Y= Y \otimes -: \cat{K} \to \cat{K}$.

\subsection*{Weight structures}{(\cite{bondarko2010weight})}
A \emph{weight structure} on a triangulated category $\cat{T}$ is a pair of  full subcategories $\cat{T}^{\leq 0}, \cat{T}^{\geq 0} \subseteq \cat{T}$ satisfying the following axioms:
\begin{enumerate}
\item $\cat{T}^{\leq 0}$ and $\cat{T}^{\geq 0}$ are idempotent-complete.
\item $\cat{T}^{\leq 0} \subseteq \cat{T}^{\leq 0}[1]$ and $\cat{T}^{\geq 0}[1] \subseteq \cat{T}^{\geq 0}$.
\item \emph{Orthogonality}: For every $X \in \cat{T}^{\leq 0}$ and $Y \in \cat{T}^{\geq 0}$, we have $\Hom_{\cat{T}}(X, Y[1])=0$.
\item \emph{Weight decomposition}: For every object $X \in \cat{T}$, there exists a distinguished triangle
\[B \to X \to A \to B,\]
with $B \in \cat{T}^{\leq 0}$ and $A \in \cat{T}^{\geq 0}[1]$.
\end{enumerate}
We denote the shifted subcategories and their intersection by
\[\cat{T}^{\leq n}:=\cat{T}^{\leq 0}[n], \quad \cat{T}^{\geq n}:=\cat{T}^{\geq 0}[n], \quad \cat{T}^{=0}:= \cat{T}^{\leq 0} \cap \cat{T}^{\geq 0},\]
and we call $\cat{T}^{=0}$ the \emph{heart} of the weight structure. A weight structure is \emph{bounded} if it exhausts the category:
\[\bigcup_{n \in \Z} \cat{T}^{\leq n} = \cat{T} = \bigcup_{n \in \Z} \cat{T}^{\geq n}.\]
In general, if a triangulated category $\cat{T}$ is equipped with a bounded weight structure, any object $X \in \cat{T}$ can be expressed as an iterative extension of finitely many graded pieces $X_i=C_i[i]$, where $C_i \in \cat{T}^{=0}$. These pieces may be arranged into a finite \emph{weight complex} for $X$, which we denote as follows:
\[\begin{tikzcd}[ampersand replacement=\&]
	{t(X):= \quad 0} \& {C_{n}} \& {C_{n-1}} \& \dots \& {C_{m+1}} \& {C_m} \& 0
	\arrow[from=1-1, to=1-2]
	\arrow["\delta_n", from=1-2, to=1-3]
	\arrow["\delta_{n-1}", from=1-3, to=1-4]
	\arrow["\delta_{m+2}", from=1-4, to=1-5]
	\arrow["\delta_{m+1}", from=1-5, to=1-6]
	\arrow[from=1-6, to=1-7]
\end{tikzcd}\]
We note that such weight complexes are in general neither unique nor functorial. The primary technique we will use to construct weight structures is provided by the following result.
\begin{thm}{(\cite[Theorem 4.3.2]{bondarko2010weight})} \label{thereom: bondarko2010weight}
\emph{Suppose that $\cat{T}$ is idempotent-complete and let $H$ be a collection of objects in $\cat{T}$ satisfying the following properties:
\begin{enumerate}\item Negativity: $\Hom_{\cat{T}}(X, Y[n])=0$ for all $X,Y \in H$ and $n \geq 1$.
\item Generation: The thick closure of $H$ is all of $\cat{T}$.
\end{enumerate}
Then there exists a unique weight structure on $\cat{T}$ whose heart $\cat{H}$ is the idempotent-completion of the additive subcategory generated by $H$.}
\end{thm}
In this article, we will consider bounded weight structures on tt-categories $\cat{K}^c$ that admit a stable symmetric monoidal $\infty$-categorical enhancement. When such a bounded weight structure is compatible with the tensor structure--meaning its heart $\cat{H}=(\cat{K}^c)^{=0}$ contains the tensor unit $\unit$ and it is closed under tensor products--then by \cite{sosnilo2019theorem} and \cite{aoki2020weight}, we may define a \emph{weight complex functor}
\[t: \cat{K}^c \to \K^b(\cat{H}) \]
which is a conservative tt-functor. 
\begin{remark}\label{remark: weight complex functors and local tt-categories}
A useful consequence of the above result is that if $\K^b(\cat{H})$ is a \emph{local} tt-category (that is, $0$ is a prime tt-ideal)--for instance, if $\cat{H}$ is semisimple--then the conservativity of the weight complex functor implies that $\cat{K}^c$ is also local.
\end{remark}
\subsection{Motivic recollections}{(\cite{voevodsky2000triangulated})}
Throughout the paper, let $k$ be a perfect field of exponential characteristic $e$, and let $R$ be a commutative ring such that $e \in R^\times$. We denote by 
\[\DM(k ; R)\] 
the \emph{big tensor-triangulated category of Voevodsky motives}. Recall that $\DM(k ; R)$ admits a stable symmetric monoidal $\infty$-categorical enhancement $\mathcal{DM}(k ; R)$ by \cite{rondigs2008modules} and \cite{hoyois2017motivic}; moreover, it is rigidly-compactly generated by the motives of smooth varieties. We denote the tensor unit by $\T := \M(\Spec(k))$ and the rigid full subcategory of compact objects by 
\[\DMgm(k ; R).\]
For any object $X$, we write $X(n)[m]$ for the twisted suspension $X \otimes \T(n)[m]$ and use the abbreviation $X\{n\} := X(n)[2n]$ for the \emph{pure shift} of weight $n$.
\subsection*{Motives of Čech simplicial schemes}{(\cite{vishik2022isotropic})}
For a smooth variety $P$ over $k$, we write $\check{C}(P)$ for its \emph{Čech simplicial scheme}, and we denote its motive by $\mathcal{X}_{P}$. The canonical projection $\check{C}(P) \to \Spec(k)$ induces a structure morphism $\mathcal{X}_{P} \to \T$, giving rise to the following idempotent triangle
\[ \mathcal{X}_{P} \to \T \to \tilde{\mathcal{X}}_{P} \to \mathcal{X}_{P}[1] \]
in $\DM(k ; R)$. Furthermore, for any two smooth varieties $P$ and $Q$ over $k$, we have the standard identifications
\[ \mathcal{X}_{P} \otimes \mathcal{X}_{Q} \cong \mathcal{X}_{P \times Q}, \quad \tilde{\mathcal{X}}_{P} \otimes \tilde{\mathcal{X}}_{Q} \cong \tilde{\mathcal{X}}_{P \sqcup Q}. \]
Recall from \cite{vishik2022isotropic} that for any two finitely generated extensions $E$ and $E'$ of $k$, with smooth models $Q$ and $Q'$ respectively, and for any prime number $p$, we write
\[ E/k \pdom E'/k \]
if $Q'_{E}$ is $p$-isotropic. If $E \pdom E'$ and $E' \pdom E$, we write $E \peq E'$. This partial ordering extends naturally to the motives of Čech simplicial schemes of smooth varieties: in particular, we write
\[ \mathcal{X}_{P} \pdom \mathcal{X}_{Q} \]
if one (and therefore all) of the following equivalent conditions is satisfied:
\begin{enumerate}
    \item $\mathcal{X}_{P} \otimes \mathcal{X}_{Q} \cong \mathcal{X}_{P}$.
    \item $\tilde{\mathcal{X}}_{P} \otimes \tilde{\mathcal{X}}_{Q} \cong \tilde{\mathcal{X}}_{Q}$.
    \item $Q$ is $p$-isotropic over the function field of every connected component of $P$. In particular, if $P$ and $Q$ are connected smooth projective varieties, then
    \[ \mathcal{X}_{P} \pdom \mathcal{X}_{Q} \iff k(P) \pdom k(Q). \]
    \item The structure map $\mathcal{X}_{P} \to \T$ factors through the structure map $\mathcal{X}_{Q} \to \T$.
\end{enumerate}
\subsection*{Pfister quadrics and Rost motives}{(\cite{vishik1997integral})}
Let $\K^M_*(k)/2$ be the \emph{Milnor K-theory modulo} $2$ of $k$, and $\alpha \in \K^M_n(k)/2$ a non-zero pure symbol of degree $n \geq 1$. We denote by $q_{\alpha}$ the associated $n$-fold \emph{Pfister form} and by $Q_{\alpha}$ the corresponding \emph{Pfister quadric}. 
By \cite{rost1998motive}, we know that 
\[\M(Q_{\alpha}) \cong \M_{\alpha}\otimes \M(\PP^{d_{n-1}}),\] 
where $d_n=2^{n}-1$. The indecomposable factor  $\M_{\alpha}$ is called a \emph{Rost motive}. Over the algebraic closure, the Rost motive splits as 
\[M_{\alpha}|_{\bar{k}} \cong \T \oplus \T\{d_{n-1}\}.\] 
Over the base field, this splitting is filtered by natural maps $\T\{d_{n-1} \} \to M_\alpha \to \T$ whose composition is zero, giving rise to the following diamond of distinguished triangles in $\DM(k ; \F_2)$,
\[\begin{tikzcd}[ampersand replacement=\&,sep=small]
	\&\& \T \\
	\\
	{\M_{\alpha}} \&\& {R_\alpha} \&\& {\tilde{\M}_\alpha} \\
	\\
	\&\& {\T\{d_{n-1}\}}
	\arrow["{[1]}", from=1-3, to=3-3]
	\arrow["{[1]}", from=1-3, to=3-5]
	\arrow[""{name=0, anchor=center, inner sep=0}, from=3-1, to=1-3]
	\arrow[from=3-3, to=3-1]
	\arrow[from=3-3, to=3-5]
	\arrow[""{name=1, anchor=center, inner sep=0}, "{[1]}", from=3-5, to=5-3]
	\arrow[from=5-3, to=3-1]
	\arrow[from=5-3, to=3-3]
	\arrow["\star"'{pos=0.6}, shift left=3, draw=none, from=0, to=3-3]
	\arrow["\star", shift right=3, draw=none, from=1, to=3-3]
\end{tikzcd}\]
where $\tilde{\M}_{\alpha}$ denotes the \emph{(completely) reduced Rost motive}. 

Let $\mathcal{X}_\alpha := \mathcal{X}_{Q_\alpha}$ be the motive of the Čech simplicial scheme of the Pfister quadric $Q_\alpha$, and let $\tilde{\mathcal{X}}_\alpha$ be its reduced motive. Recall that $\mathcal{X}_\alpha \otimes \M_\alpha \cong \M_\alpha$ and $\tilde{\mathcal{X}}_\alpha \otimes \tilde{\M}_\alpha \cong \tilde{\M}_\alpha$. 
Consequently, from the idempotent triangle of $\mathcal{X}_\alpha$ we deduce that $\M_{\alpha} \otimes \tilde{\M}_\alpha \cong 0$. Tensoring the above diamond by $\mathcal{X}_\alpha$ and $\tilde{\mathcal{X}}_\alpha$ respectively, yields the following two distinguished triangles in $\DM(k ; \F_2)$:
\[\mathcal{X}_\alpha\{d_{n-1}\} \to \M_{\alpha} \to \mathcal{X}_\alpha \to \mathcal{X}_\alpha(d_{n-1})[2d_{n-1}+1];\] 
\[\tilde{\mathcal{X}}_\alpha\{d_{n-1}\} \to \tilde{\mathcal{X}}_\alpha[-1] \to \tilde{\M}_{\alpha} \to \tilde{\mathcal{X}}_\alpha(d_{n-1})[2d_{n-1}+1].\]  

\section{Mixed motives of quadrics}
We introduce the main categories of interest of this article.
\begin{defn}
The big tt-category of \emph{mixed motives of quadrics} 
\[\DQM(k;R)\] 
is defined as the full localising subcategory of $\DM(k, R)$ generated by all twists of the motives of smooth projective quadrics over $k$. We denote by 
\[\DQMgm(k;R)\] 
its full tt-subcategory of compact objects.
\end{defn}
From now on, we restrict our attention to the case of primary interest, $R=\F_2$. Nevertheless, it is clear that the following definitions readily generalise to more general coefficient rings.
\begin{defn}
For any non-zero pure symbol $\alpha \in \K^M_*(k)/2$, we denote by
\[\DQM_{\tilde{\mathcal{X}}_\alpha}(k ; \F_2)\] 
the full localising subcategory of $\DQM(k ; \F_2)$ consisting of objects that are annihilated upon tensoring with $\mathcal{X}_\alpha$. Furthermore, we define 
\[\DTM_{\tilde{\mathcal{X}}_\alpha}(k ; \F_2)\] 
to be the full localising subcategory of $\DQM_{\tilde{\mathcal{X}}_\alpha}(k ; \F_2)$ generated by twists of $\tilde{\mathcal{X}}_\alpha$.
\end{defn}
Observe that $\tilde{\mathcal{X}}_\alpha$ is a right idempotent, appearing in the idempotent triangle
\[\mathcal{X}_\alpha \to \T \to \tilde{\mathcal{X}}_\alpha \to \mathcal{X}_\alpha[1].\]
In particular $\DQM_{\tilde{\mathcal{X}}_\alpha}(k ; \F_2)= \tilde{\mathcal{X}}_\alpha \otimes \DQM(k ; \F_2)$ is a smashing localisation, whose kernel is the localising tensor ideal generated by the Rost motive $\M_\alpha$. This implies that $\DQM_{\tilde{\mathcal{X}}_\alpha}(k ; \F_2)$ is a rigidly-compactly generated tt-category, and its compact generators are twists of $\tilde{\mathcal{X}}_\alpha \otimes \M(Q)$, where $Q$ is a smooth projective quadric over $k$. We denote by 
\[\DQM^{\geom}_{\tilde{\mathcal{X}}_\alpha}(k ; \F_2)\] 
the full tt-subcategory of compact objects, and 
\[\DTM^{\geom}_{\tilde{\mathcal{X}}_\alpha}(k ; \F_2)\] 
the full tt-subcategory generated by twists of $\tilde{\mathcal{X}}_\alpha$.

\subsection{Mixed motives of quadrics over real closed fields}
From now on we shall denote by $F$ a real closed field. 

Recall that $\K^M_*(F)/2$ is generated in degree $1$ by the pure symbol $\rho=\{-1\}$, and it is isomorphic to the polynomial ring $\F_2[\rho]$.

The first advantage of working over real closed fields is that Tate motives and Rost motives define a set of compact generators for $\DQM(F ; \F_2)$: this follows from the fact that over real closed fields, anisotropic quadrics are \emph{excellent} \cite{vishik2011excellent, rost1990some}.

We now begin to define the geometric tt-functors that will constitute our jointly conservative family. First, we shall denote by $Q_1$ the zero-dimensional Pfister quadric
\[\{x^2+y^2=0\}\subset \mathbb{P}_F^2,\]
and we denote by $\mathcal{X}_1$ the motive of its Čech simplicial scheme, and by $\tilde{\mathcal{X}}_1$ its reduced motive. Since $F$ is real closed, then the function field $F(Q_1) \cong \overline{F}$.
\begin{lemma}[Definition of $\pi_1$] \label{lemma: definition of pi1}
\emph{The restriction of the base-change functor $\DM(F ; \F_2) \to \DM(\overline{F}; \F_2)$
to $\DQM(F ; \F_2)$ defines a geometric tt-functor
\[\pi_1: \DQM(F ; \F_2) \to \DTM(\overline{F}; \F_2) \] }  
\end{lemma}
\begin{proof}
Since the motives of quadrics split into direct sums of Tate motives over the algebraic closure, the essential image of $\pi_1$ is contained in $\DTM(\overline{F}; \F_2)$. Because coproducts in $\DQM(F ; \F_2)$ are computed in the ambient category $\DM(F ; \F_2)$ and the base-change functor admits a right adjoint, $\pi_1$ is a geometric tt-functor.
\end{proof}
Next, we shall denote by $Q_2$ the imaginary conic defined by
\[\{x^2+y^2+z^2=0\} \subset \mathbb{P}_F^2,\]
and we denote by $\mathcal{X}_2$ the motive of its Čech simplicial scheme, and by $\tilde{\mathcal{X}}_2$ its reduced motive.
\begin{lemma}[Definition of $\pi_2$]\label{lemma: definition of pi2}
\emph{Let $\rho=\{-1\} \in \K_1^M(F(Q_{2}))/2$, then the restriction of the composite
\[\DM(F ; \F_2) \to \DM(F(Q_{2}); \F_2) \to \tilde{\mathcal{X}}_{\rho} \otimes \DM(F(Q_{2}); \F_2),\] 
to $\DQM(F ; \F_2)$ defines a geometric tt-functor}
\[\pi_2: \DQM(F ; \F_2) \to \DTM_{\tilde{\mathcal{X}}_{\rho}}(F(Q_{2}); \F_2).\]
\end{lemma}
\begin{proof}
First, it is clear that $\pi_2$ is geometric since coproducts in $\DQM(F ; \F_2)$ are computed in the ambient category $\DM(F ; \F_2)$, and the defining functor is a composition of functors admitting right adjoints. Moreover,
from the equivalence 
\[\tilde{\mathcal{X}}_{\rho} \cong \tilde{\mathcal{X}}_1|_{F(Q_{2})},\] 
we see that the following diagram commutes. 
\[\begin{tikzcd}[ampersand replacement=\&]
	{\DM(F ; \F_2)} \&\& { \DM(F(Q_{2}); \F_2)} \\
	{\tilde{\mathcal{X}}_1\otimes\DM(F ; \F_2)} \&\& {\tilde{\mathcal{X}}_\rho\otimes\DM(F(Q_{2}); \F_2)}
	\arrow["{(-)|_{F(Q_{2})}}", from=1-1, to=1-3]
	\arrow["{\tilde{\mathcal{X}}_1\otimes -}"', from=1-1, to=2-1]
	\arrow["{\tilde{\mathcal{X}}_\rho\otimes -}", from=1-3, to=2-3]
	\arrow["{(-)|_{F(Q_{2})}}", from=2-1, to=2-3]
\end{tikzcd}\] 
From this we may observe that
\[\pi_2(\M_n)=\begin{cases}
    0 &  n=1\\
    \tilde{\mathcal{X}}_{\rho} \oplus  \tilde{\mathcal{X}}_{\rho} \{d_{n-1}\} &  n \geq 2 
\end{cases}\]
from which it follows that the essential image of $\pi_2$ is contained in $\DTM_{\tilde{\mathcal{X}}_{\rho}}(F(Q_{2}); \F_2)$.
\end{proof}
For a pure symbol $\alpha=\rho^n$ with $n \geq 3$, we shall denote the associated Pfister quadric by $Q_{n}$, the motive of its Čech simplicial scheme by $\mathcal{X}_n$, and its completely reduced motive by $\tilde{\mathcal{X}}_n$. Observe that since the imaginary conic $Q_2$ is a Pfister neighbour of the Pfister quadric $Q_{\rho^2}$, then 
\[\mathcal{X}_2:= \mathcal{X}_{Q_2} \cong \mathcal{X}_{\rho^2}, \quad  \tilde{\mathcal{X}}_2:= \tilde{\mathcal{X}}_{Q_2} \cong \tilde{\mathcal{X}}_{\rho^2}.\]
Furthermore, for any $1 \leq n \leq m$, we have that $\mathcal{X}_n \twodom \mathcal{X}_m$. 
\begin{lemma}[Definition of $\pi_n$ for $n \geq 3$]\label{lemma: definition of pin}
\emph{For $n \geq 3$, let $\rho^{n-1} \in \K_{n-1}^M(F(Q_{n}))/2$, then the restriction of the composite
\[\DM(F ; \F_2) \to \DM(F(Q_{n}); \F_2) \to \tilde{\mathcal{X}}_{\rho^{n-1}} \otimes \DM(F(Q_{n}); \F_2),\]
to $\DQM(F ; \F_2)$ defines a geometric tt-functor}
\[\pi_n: \DQM(F ; \F_2) \to \DTM_{\tilde{\mathcal{X}}_{\rho^{n-1}}}(F(Q_{n}); \F_2).\]
\end{lemma}
\begin{proof}
The fact that $\pi_n$ is geometric follows from the same reasoning as in the proof of the previous lemma. Similarly, because 
\[\tilde{\mathcal{X}}_{\rho^{n-1}} \cong \tilde{\mathcal{X}}_{n-1}|_{F(Q_{n})},\] 
then from the commutativity of the following diagram
\[\begin{tikzcd}[ampersand replacement=\&]
	{\DM(F ; \F_2)} \&\& { \DM(F(Q_n); \F_2)} \\
	{\tilde{\mathcal{X}}_{n-1}\otimes\DM(F ; \F_2)} \&\& {\tilde{\mathcal{X}}_{\rho^{n-1}}\otimes\DM(F(Q_{n}); \F_2)}
	\arrow["{(-)|_{F(Q_n)}}", from=1-1, to=1-3]
	\arrow["{\tilde{\mathcal{X}}_{n-1}\otimes -}"', from=1-1, to=2-1]
	\arrow["{\tilde{\mathcal{X}}_{\rho^{n-1}}\otimes -}", from=1-3, to=2-3]
	\arrow["{(-)|_{F(Q_{n})}}", from=2-1, to=2-3]
\end{tikzcd}\]
we may deduce the following.
\[\pi_n(\M_k)=\begin{cases}
    0 &  k \leq n-1\\
    \tilde{\mathcal{X}}_{\rho^{n-1}} \oplus  \tilde{\mathcal{X}}_{\rho^{n-1}} \{d_{k-1}\} &  k \geq n 
\end{cases}\]
In particular, the essential image of $\pi_n$ is contained in $\DTM_{\tilde{\mathcal{X}}_{\rho^{n-1}}}(F(Q_{n}); \F_2)$.
\end{proof}
\subsection{Function fields of real anisotropic quadrics} 
A key feature of anisotropic quadrics over real closed fields is that the Milnor K-theory (modulo 2) of their function fields vanishes in degrees exceeding their dimension. 
\begin{lemma}\label{lemma: function fields of real closed anisotropic quadrics}
\emph{Let $Q$ be an anisotropic quadric over $F$ of dimension $m$. Then $\K^M_n(F(Q))/2=0$ for all $n \geq m+1$.}
\end{lemma}
\begin{proof}
Let $I \subseteq W(F(Q))$ denote the fundamental ideal of the Witt ring of $F(Q)$. By the Milnor conjecture on quadratic forms \cite{orlov2007exact}, it suffices to show that $I^n=0$ for all $n \geq m+1$. Since $F(Q)$ is an extension of $F$ of transcendence degree $m$, then by the Tsen-Lang theorem \cite[Chapter XI]{lam2005introduction} any quadratic form over $F(Q)(\sqrt{-1})$ of dimension strictly greater than $2^m$ is isotropic. By the Elman-Lam-Krüskemper theorem \cite[Theorem 1.3]{becher2011elman}, this bounds the torsion part of the fundamental powers, forcing $(I^n)^{tor}=0$ for all $n \geq m+1$. However,  $-1$ is a sum of squares in $F(Q)$. As a result, every quadratic form over $F(Q)$ is torsion in the Witt ring; in particular, $I^n = (I^n)^{tor}$ for all $n$. Combining these facts, we conclude that $I^n = (I^n)^{tor} = 0$ for all $n \geq m+1$, as required.
\end{proof}
\subsection{Real isotropic category of quadrics}\label{section: real isotropic category of quadrics}
We now introduce a smashing localisation of the category of mixed motives of quadrics over real closed fields. In combination with the functors $\pi_n$ defined previously, this will complete our jointly conservative family of geometric tt-functors.

Recall that $\mathcal{X}_n \twodom \mathcal{X}_{n+1}$, in particular this implies \cite[Definition 2.1]{vishik2022isotropic} that the structure map $\mathcal{X}_{n} \to \T$ factors through the structure map $\mathcal{X}_{n+1} \to \T$. In the stable $\infty$-categorical enhancement $\mathcal{DM}(F ; \F_2)$ we define the sequential colimits
\[\mathscr{Y}^{iso} =\colim_n{(\mathcal{X}_n)}, \quad \Upsilon^{iso}= \mathrm{cofib}(\mathscr{Y}^{iso} \to \T) \simeq \colim_n(\mathrm{cofib}(\mathcal{X}_n \rightarrow \T)).\]
Because $\mathcal{X}_n \otimes \mathcal{X}_n \simeq \mathcal{X}_n$ and $\otimes$ commutes with colimits, we deduce that 
\[\mathscr{Y}^{iso} \otimes \mathscr{Y}^{iso} \simeq \mathscr{Y}^{iso}, \quad \Upsilon^{iso} \otimes \Upsilon^{iso} \simeq \Upsilon^{iso}.\]
\begin{remark}\label{remark: Tanania real isotropic sphere spectrum}
We note that $\Upsilon^{iso}$ is by definition the motivic realisation of the real isotropic sphere spectrum in $\mathrm{SH}(F)$ defined by Tanania in \cite{tanania2025real} as 
\[\unit^{iso} := \colim_n(\mathrm{cofib}(\check{C}(Q_n) \to \unit)).\]
In particular, for the reasons outlined in \cite[Remark 3.2]{tanania2025real}, this object is not the same as the isotropic projector in the sense of \cite{vishik2022isotropic}. 
\end{remark}
In $\mathcal{DM}(F ; \F_2)$, the fiber sequence $\mathscr{Y}^{iso} \to \T \to\Upsilon^{iso}$
gives rise to an idempotent triangle 
\[\mathscr{Y}^{iso} \to \T \to \Upsilon^{iso} \to \mathscr{Y}^{iso}[1] \]
in the homotopy category $\DM(F ; \F_2)$. Moreover, $\mathscr{Y}^{iso}$ and $\Upsilon^{iso}$ are homotopy colimits of objects in $\DQM(F ; \F_2)$; in particular, they are idempotent objects of this category.
\begin{defn}
We shall denote by $\DTM(F ; \F_2)^{iso}$ the big tt-category $\Upsilon^{iso} \otimes \DQM(F ; \F_2)$.
\end{defn}

\begin{remark}
As we shall see in the following lemma, $\Upsilon^{iso}$ is a right idempotent that annihilates the motives of all anisotropic quadrics over $F$, in particular $\DTM(F ; \F_2)^{iso}$ identifies with the full localising subcategory of $\DQM(F ; \F_2)$ generated by twists of $\Upsilon^{iso}$.

It is important to highlight that, as we have already mentioned in Remark \ref{remark: Tanania real isotropic sphere spectrum}, $\DTM(F ; \F_2)^{iso}$ is not the same as Vishik's tt-category of isotropic Tate motives $\DTM(F/F; \F_2)$ in the sense of \cite{vishik2022isotropic}.
\end{remark}

\begin{lemma}
\emph{The kernel of $\pi_\infty := \Upsilon^{iso} \otimes - \colon \DQM(F ; \F_2) \to \DTM(F ; \F_2)^{iso}$ is  }
\[ \Loc_{\otimes}(\M_n : n \in \N). \]
\end{lemma}

\begin{proof}
Since $\ker(\pi_\infty)$ is a localising tensor-ideal and $\tilde{\mathcal{X}}_n \otimes \M_m = 0$ for all $m \leq n$, the inclusion 
\[ \mathrm{Loc}_{\otimes}(\M_n : n \in \N) \subseteq \ker(\pi_\infty) \]
is clear. For the reverse inclusion, suppose that $\Upsilon^{iso}\otimes A=0$. Then $\mathscr{Y}^{iso} \otimes A \cong A$; thus, it suffices to show that $\mathscr{Y}^{iso} \in \mathrm{Loc}_{\otimes}(\M_n : n \in \N)$. This is also clear since $\mathcal{X}_k \in \mathrm{Loc}_{\otimes}(\M_n : n \in \N)$ for all $k \geq 1$.
\end{proof}

\subsection{A jointly conservative family of tt-functors}
We have defined a family of geometric functors 
\[\{\pi_n\}_{n \in \N \cup \{\infty\}}.\] 
In this section, we will prove that this family is jointly conservative. We begin by recording, for future references, how these functors map Rost motives, (completely) reduced Rost motives, $\cone(\rho)$ and $\cone(\tau)$.

\begin{table}[htbp]
    \centering
    \resizebox{\textwidth}{!}{%
    \begin{tabular}{@{} |l| c |c| c| c| @{}}
        \toprule
        $\pi_k$ & $\M_n$ & $\tilde{\M}_n$ & $\cone(\rho)$ & $\cone(\tau)$  \\
        \midrule      
        $k=1$ & $ \T|_{\overline{F}} \oplus \T|_{\overline{F}}\{d_{n-1}\}$ & $0$  & $\T|_{\overline{F}} \oplus \T|_{\overline{F}}(1)[1]$  & $\cone(\tau)|_{\overline{F}}$ \\[3ex]
        
        $k=2$ & $ \begin{cases}
         0 &  n=1  \\
        \tilde{\mathcal{X}}_{\rho} \oplus  \tilde{\mathcal{X}}_{\rho} \{d_{n-1}\} &  n \geq 2 
        \end{cases}$ & $\begin{cases}
        (\tilde{\M}_1)|_{F(Q_{2})} & n=1  \\
        0 &  n \geq 2
        \end{cases}$ & $\tilde{\mathcal{X}}_{\rho} \oplus  \tilde{\mathcal{X}}_{\rho}(1)[1]$  & $\tilde{\mathcal{X}}_{\rho} \oplus  \tilde{\mathcal{X}}_{\rho}(1)$ \\[3ex]

        $3 \leq k < \infty$ & $ \begin{cases}
         0 &  n \leq k-1  \\
        \tilde{\mathcal{X}}_{\rho^{k-1}} \oplus  \tilde{\mathcal{X}}_{\rho^{k-1}}\{d_{n-1}\} &  n \geq k 
        \end{cases}$ & $\begin{cases}
        \tilde{\mathcal{X}}_{\rho^{k-1}} \otimes (\tilde{\M}_n)|_{F(Q_k)}\neq 0 &  n\leq k-2  \\
        (\tilde{\M}_n)|_{F(Q_k)} & n=k-1 \\
        0 &  n \geq k
        \end{cases}$ & $\tilde{\mathcal{X}}_{\rho^{k-1}} \oplus  \tilde{\mathcal{X}}_{\rho^{k-1}}(1)[1]$  & $ \tilde{\mathcal{X}}_{\rho^{k-1}} \oplus  \tilde{\mathcal{X}}_{\rho^{k-1}}(1)$ \\[3ex]

        $k =\infty$ & $0$ &  $\tilde{\M}_n^{iso} \neq 0$ & $\cone(\rho)^{iso}  \neq 0$  & $\cone(\tau)^{iso} \neq 0$ \\[1ex]
        \bottomrule
    \end{tabular}%
    }
    \vspace{0.1cm}
    \caption{We record how each geometric tt-functor maps Rost motives, (completely) reduced Rost motives, $\cone(\rho)$ and $\cone(\tau)$.}
    \label{tab:geometric_functors}
\end{table}
\begin{thm}\label{theorem: joinlty conservative geo functors}
\emph{The family of geometric tt-functors $\{\pi_n\}_{n \in \N \cup \{\infty\}}$ is jointly conservative.}
\end{thm}
\begin{proof}
Recall from \cite[Theorem 2.3.5]{vishik1997integral} that if $P$ is a smooth projective variety over $F$ and $E=F(P)$, the kernel of the base-change functor $(-)_E \colon \DM(F ; \F_2) \to \DM(E; \F_2)$ coincides with the kernel of $\mathcal{X}_P \otimes -$. This allows us to identify the kernels of our functors as follows:
\begin{enumerate}
    \item $\ker(\pi_1) = \ker(\mathcal{X}_1 \otimes -)$;
    \item $\ker(\pi_2) = \ker(\tilde{\mathcal{X}}_1 \otimes \mathcal{X}_2 \otimes -)$;
    \item $\ker(\pi_n) = \ker(\tilde{\mathcal{X}}_{n-1} \otimes \mathcal{X}_n \otimes -)$ for all $n \geq 3$.
\end{enumerate}

Suppose $A \in \bigcap_{n \in \N \cup \{\infty\}} \ker(\pi_n)$. We first show by induction that $\mathcal{X}_n \otimes A \cong 0$ for all $n \in \N$. For the base case $n=1$, this holds by assumption. For the inductive step, if $\mathcal{X}_n \otimes A \cong 0$, then $\tilde{\mathcal{X}}_n \otimes A \cong A$, so $0 \cong \tilde{\mathcal{X}}_n \otimes \mathcal{X}_{n+1} \otimes A \cong \mathcal{X}_{n+1} \otimes A$.

Finally, since $\pi_\infty(A)=0$, then $\Upsilon^{iso} \otimes A \cong 0$, which means that $\mathscr{Y}^{iso} \otimes A \cong A$. However, $\mathcal{X}_n \otimes A \cong 0$ for all $n \in \N$, hence $0 \cong \mathscr{Y}^{iso} \otimes A$. 
\end{proof}

By Theorem \ref{theorem: barthel2024surjectivity}, the above result reduces the computation of the points of $\Spc(\DQMgm(F ; \F_2))$ to the computation of the Balmer spectra of the following cellular tt-categories:
\begin{enumerate}
    \item $\cat{K}_1^c = \DTMgm(F(Q_1); \F_2)=\DTMgm(\overline{F}; \F_2)$;
    \item $\cat{K}_2^c = \DTM^{\geom}_{\tilde{\mathcal{X}}_{\rho}}(F(Q_{2}); \F_2)$;
    \item $\cat{K}_n^c = \DTM^{\geom}_{\tilde{\mathcal{X}}_{\rho^{n-1}}}(F(Q_n); \F_2)$ for $n \geq 3$;
    \item $\cat{K}_\infty^c = \DTM^{\geom}(F ; \F_2)^{iso}$.
\end{enumerate}

We recall that the computation of $\Spc(\cat{K}_1^c)$ was obtained by Gallauer \cite{gallauer2019tt}, who showed it is homeomorphic to the following two-point Sierpiński space:
\[\begin{tikzcd}[ampersand replacement=\&]
	0 \\
	{\langle \cone(\tau)|_{\overline{F}} \rangle}
	\arrow[no head, from=1-1, to=2-1]
\end{tikzcd}\]
\section{The Balmer spectrum of some cellular motivic categories}
In this section, we compute the Balmer spectra of certain cellular motivic tt-categories, which will subsequently allow us to deduce the computations of $\Spc(\cat{K}_n^c)$ for $n \geq 2$. Let $k$ be an arbitrary field (not necessarily real closed), and let $\alpha \in \K^M_m(k)/2$ be a non-zero pure symbol with $m \geq 1$. We begin by considering the tt-category $\DTM_{\tilde{\mathcal{X}}_\alpha}(k; \F_2)$. 

Recall from \cite[Theorem 3.5]{vishik2022isotropic} that the bi-graded endomorphism ring of the unit $\tilde{\mathcal{X}}_\alpha$ is given by
\[\mathbb{E}\mathrm{nd}(\tilde{\mathcal{X}}_\alpha) = S_*^\alpha[r_0, \dots, r_{m-1}]\big /(\{r_{i}^2-r_{i+1}\rho\}_{0 \leq i \leq m-2}),\]
where $S^\alpha_*$ denotes the graded ring $(\K^M_*(k)/2)/\ker(\alpha \cdot)$, and the generators $r_i$ are morphisms \[r_i : \tilde{\mathcal{X}}_{\alpha} \to \tilde{\mathcal{X}}_{\alpha}(-d_{i})[-2d_{i}-1],\] 
where $d_i=2^{i}-1$.
\begin{remark}\label{remark: support of End tilde Xalpha}
It is convenient to visualise $\mathbb{E}\mathrm{nd}(\tilde{\mathcal{X}}_\alpha)$ inside the lattice $\Z^2$ by placing a homogeneous element of bi-degree $(a)[b]$ at the lattice point $(a,b)\in \Z^2$. With this convention, the generators $r_i$ have bi-degrees $(-d_i,\,-2d_i-1)$, hence each $r_i$ lies strictly below the Chow line $y=2x$. Since $S_*^\alpha$ lies on the diagonal (in bi-degrees $(n)[n]$), it follows that $\mathbb{E}\mathrm{nd}(\tilde{\mathcal{X}}_\alpha)$ is concentrated below the Chow line $y=2x$ of the lattice $\Z^2$.
\end{remark}

\begin{lemma}\label{lemma: chow weight structure and locality of cellular tt-categories}
\emph{The category $\DTM^{\geom}_{\tilde{\mathcal{X}}_\alpha}(k; \F_2)$ affords a unique bounded Chow weight structure whose heart $\cat{H}$ is the additive subcategory generated by the pure shifts of the unit,
\[\tilde{\mathcal{X}}_\alpha\{i\}, \quad i \in \Z.\]
In particular, $\DTM^{\geom}_{\tilde{\mathcal{X}}_\alpha}(k; \F_2)$ is a local tt-category.}
\end{lemma}
\begin{proof}
By Remark \ref{remark: support of End tilde Xalpha}, for any $i,j \in \Z$ we know that
\[\Hom(\tilde{\mathcal{X}}_\alpha\{i\}, \tilde{\mathcal{X}}_\alpha\{j\})= \begin{cases}
\F_2 & i=j,\\
0 & i\neq j.
\end{cases} \]
From this it is clear that $\cat{H}$ is semisimple, and in particular idempotent-complete. Because the thick closure of $\cat{H}$ generates the entire category $\DTM^{\geom}_{\tilde{\mathcal{X}}_\alpha}(k; \F_2)$, then by Theorem \ref{thereom: bondarko2010weight} it suffices to verify the negativity condition. This reduces to checking that 
\[\mathrm{Hom}(\tilde{\mathcal{X}}_\alpha, \tilde{\mathcal{X}}_\alpha(j)[2j+n]) = 0 \quad \text{for } j \in \Z \text{ and } n \geq 1,\]
which follows again from Remark \ref{remark: support of End tilde Xalpha}. For the second claim, observe that $\cat{H}$ is a tensor and semisimple category. Furthermore, 
\[\DTM^{\geom}_{\tilde{\mathcal{X}}_\alpha}(k; \F_2)\] 
naturally admits a stable symmetric monoidal $\infty$-categorical enhancement since it can be modelled as the subcategory of compact objects of a smashing localisation of a stable symmetric monoidal $\infty$-category. The rest follows from Remark \ref{remark: weight complex functors and local tt-categories}.
\end{proof}

\subsection{Computation of $\Spc(\cat{K}^c_2)$} \label{section: Spc of Kc2}
Let $\alpha \in \K_1^M(k)/2$ be a non-zero pure symbol of degree $1$, and consider 
\[\cat{L}_2=\DTM_{\tilde{\mathcal{X}}_\alpha}(k; \F_2).\]
In this section, we shall work under the following assumption.
\begin{assumption} \label{assumption: Spc of Kc2}
$S^\alpha_l=0$ for all $l \geq 1$. This implies that $\mathbb{E}\mathrm{nd}(\tilde{\mathcal{X}}_\alpha)=\F_2[r_0].$  
\end{assumption}
\begin{lemma}
\emph{The tt-category $\cat{K}_2 = \DTM_{\tilde{\mathcal{X}}_{\rho}}(F(Q_{2}); \F_2)$ satisfies Assumption \ref{assumption: Spc of Kc2}}.
\end{lemma}
\begin{proof}
By Lemma \ref{lemma: function fields of real closed anisotropic quadrics}, we know that $\K_l^{M} (F(Q_2)/2=0$ for $l \geq 2$. Since $\rho$  is an element of degree $1$, then it follows immediately that $S^\alpha_l=0$ for all $l \geq 1$.
\end{proof}
To simplify notation, we shall denote the tensor unit $\tilde{\mathcal{X}}_{\alpha}$ of $\cat{L}_2$ by $\unit$. By the previous lemma, the computation of $\Spc(\cat{K}^c_2)$ follows from the next result.
\begin{thm}\label{thm: case m=1}
\emph{Under Assumption \ref{assumption: Spc of Kc2}, the Balmer spectrum of $\cat{L}^c_2$ is homeomorphic to the following two-point Sierpiński space.}
\[\begin{tikzcd}[ampersand replacement=\&]
	0 \\
	{\langle \cone(r_0) \rangle}
	\arrow[no head, from=1-1, to=2-1]
\end{tikzcd}\]
\end{thm}
\begin{proof}
We have already established in  Lemma \ref{lemma: chow weight structure and locality of cellular tt-categories} that $0$ is prime. To prove that $\langle \cone(r_0) \rangle$ is the only non-zero prime ideal we will establish the following.
\begin{enumerate}
    \item $\supp(\cone(r_0))=\{0\}$.
    \item $U(\cone(r_0))=\{\langle \cone(r_0) \rangle\}$.
\end{enumerate}
From this the description of the points is immediate; moreover, the specialisation relations are clear and the closed subsets are the specialisation-closed ones.

To prove the first claim, we will show in Proposition \ref{prop: case m=2 support of cone(r_0)}  that for any non-zero object $A$, $\cone(r_0) \in \langle A \rangle$, from which (1) immediately follows.

For the second claim, we begin by formally inverting the morphism $r_0$. Recall from the proof of Lemma \ref{lemma: chow weight structure and locality of cellular tt-categories} that $\cat{L}_2$ admits a stable symmetric monoidal $\infty$-categorical enhancement $\mathcal{L}_2$. In $\mathcal{L}_2$ (writing $\Sigma$ for the suspension), we form the sequential colimit
\[\tilde{\mathcal{D}}_{r_0}:=\colim_n(\unit \xrightarrow{r_0} \Sigma^{-1}\unit \xrightarrow{\Sigma^{-1}r_0} \Sigma^{-2}\unit \to \dots ),\]
with canonical map $\eta: \unit \to \tilde{\mathcal{D}}_{r_0}$. Since $\otimes$ preserve colimits, then $\tilde{\mathcal{D}}_{r_0}$ is an idempotent. Setting $\mathcal{D}_{r_0}:=\mathrm{fib}(\eta)$, we obtain a fiber sequence of complementary idempotents $\mathcal{L}_2$
\[\mathcal{D}_{r_0} \to \unit \xrightarrow{\eta} \tilde{\mathcal{D}}_{r_0}.\]
Passing to the homotopy category $\cat{L}_2$, the above fiber sequence yields the associated idempotent triangle
\[\mathcal{D}_{r_0} \to \unit \to \tilde{\mathcal{D}}_{r_0} \to \mathcal{D}_{r_0}[1].\]
Finally, tensoring by the right idempotent $\tilde{\mathcal{D}}_{r_0}$ defines a smashing localisation of $\cat{L}_2$ whose kernel is $\Loc_{\otimes}(\cone(r_0))$: one inclusion is clear since $\tilde{\mathcal{D}}_{r_0} \otimes \cone(r_0) \cong 0$; for the other inclusion, note that $\tilde{\mathcal{D}}_{r_0} \otimes A \cong 0 $ implies that  $\mathcal{D}_{r_0} \otimes A \cong A$, and this belongs to $\Loc_{\otimes}(\cone(r_0))$ since 
\[\mathcal{D}_{r_0} \cong \hocolim_n\big(\cone(\unit \xrightarrow{r_0^n} \unit[-n])\big)[-1] \in \Loc_{\otimes}(\cone(r_0)).\] 
This means that the tt-category
\[\cat{M}_2=\tilde{\mathcal{D}}_{r_0} \otimes \cat{L}_2\]
gets identified with the localisation of $\cat{L}_2$ at the localising tensor ideal generated by $\cone(r_0)$. This reduces (2) to proving that $\Spc(\cat{M}_2^c)=\{0\}$. To show this, first observe that from the idempotent triangle of $\tilde{\mathcal{D}}_{r_0}$ in $\cat{L}_2$, we can deduce that 
\[\mathbb{E}\mathrm{nd}(\tilde{\mathcal{D}}_{r_0})=\colim_n(\mathbb{E}\mathrm{nd}(\unit)),\]
where the transition maps are given by multiplication by $r_0$. In particular, the above ring is isomorphic to the graded field $\F_2[r_0, r_0^{-1}]$. Because $\cat{M}_2$ is generated by $\tilde{\mathcal{D}}_{r_0}(i)$ for $i \in \Z$, then any object $\cat{M}_2^c$ is isomorphic to a finite direct sum of $\tilde{\mathcal{D}}_{r_0}(i)$ for $i \in \Z$. From this, it is clear that $\Spc(\cat{M}_2^c)=\{0\}$. 
\end{proof}

\begin{prop}\label{prop: case m=2 support of cone(r_0)}
\emph{For every non-zero object $A \in \cat{L}_2^c$, $\cone(r_0) \in \langle A \rangle.$}
\end{prop}
\begin{proof}
The existence of a bounded weight structure with semisimple heart ensures that any object admits a weight filtration for which the weight complex has zero differentials \cite[Proposition 5.10]{vishik2024isotropic}. If $A$ is a non-zero object,  we have a distinguished triangle,
\[\begin{tikzcd}[ampersand replacement=\&, sep=small]
	{A_{>m}[-1]} \& {A_m} \& A \& {A_{>m}.}
	\arrow["f", from=1-1, to=1-2]
	\arrow[from=1-2, to=1-3]
	\arrow[from=1-3, to=1-4]
\end{tikzcd}\]
where $A_m$ is a non-empty direct sum of twisted suspensions of the unit. In particular, to prove the proposition it is sufficient to show that $f$ is $\otimes$-nilpotent on $\cone(r_0)$. Because the weight complex of $A$ was chosen with zero differentials, the morphism $f$ factors through an object in $(\cat{L}^c_2)^{\geq m+1}$. Hence, by duality it suffices to prove that for any object $Y \in (\cat{L}^c_2)^{\geq 1}$, any morphism $s:Y \to \unit$ is $\otimes$-nilpotent on $\cone(r_0)$. We will do this by proving the stronger statement that 
\[ \Hom(\cone(r_0)(i)[j], \cone(r_0))=0, \quad \text{ for } j-2i \geq 1. \]
This is clearly sufficient since $Y$ is an extension of $\unit(i)[j]$ for $i$ and $j$ in the above range. Applying $\Hom(-, \cone(r_0))$ to the defining triangle of $\cone(r_0)(i)[j]$, we observe that it is sufficient to check that 
\[ \text{(I)} \, \, \Hom(\unit(i)[j-1], \cone(r_0))=0, \quad \text{(II)} \, \, \Hom(\unit(i)[j+1],  \cone(r_0))=0 .\]
Moreover, applying $\Hom(\unit(i)[j'], -)$ to the defining triangle of $\cone(r_0)$, and using the fact that multiplication by $r_0$ is injective in $\mathbb{E}\mathrm{nd}(\unit)$, we deduce that 
\[ \Hom(\unit(i)[j'], \cone(r_0))\cong \Hom(\unit(i)[j'], \unit[-1])/(\im(r_0)_*).\]
The term on the right can be computed by modding out $\mathbb{E}\mathrm{nd}(\unit)$ by the ideal generated by $r_0$, which is isomorphic to $\F_2$ concentrated in bi-degree $(0,0)$. Since the condition $j-2i \geq 1$ implies we are away from degree $(0,0)$, we see that (I) and (II) vanish.
\end{proof}

\subsection{Computation of $\Spc(\cat{K}^c_n)$ for $ 3\leq n < \infty$ } \label{section: Spc of Kcm}
Let $\alpha \in \K_m^M(k)/2$ be a non-zero pure symbol of degree $m \geq 2$, and consider 
\[\cat{L}_m := \DTM_{\tilde{\mathcal{X}}_\alpha}(k; \F_2)\] 
In this section, we shall work under the following assumption.
\begin{assumption}\label{assumption: computation m geq 2}
$S^\alpha_l = 0$ for all degrees $l \geq N$, where $N$ is some sufficiently large integer constant. Moreover, we will also assume that $\rho= 0$ in $S^\alpha_*$. Under these conditions, the bi-graded endomorphism ring of the unit simplifies as follows 
\[\mathbb{E}\mathrm{nd}(\tilde{\mathcal{X}}_\alpha) = S_*^\alpha[r_0, \dots, r_{m-1}] \big/ (r_i^2)_{0 \leq i \leq m-2}.\]
\end{assumption}
\begin{remark}\label{remark: support of the bigraded endomorphism ring of the unit for m geq 2}
Recall from Remark \ref{remark: support of End tilde Xalpha} that we may visualise $\mathbb{E}\mathrm{nd}(\tilde{\mathcal{X}}_\alpha)$ inside the lattice $\Z^2$. Assumption \ref{assumption: computation m geq 2} forces the morphisms $r_i$ for $0 \leq i \leq m-2$ to be nilpotent, so the only source of unboundedness comes from the polynomial generator $r_{m-1}$. Writing $d:=d_{m-1}$, the powers of $r_{m-1}$ lie on the line $y=(2+1/d)x$ and $\mathbb{E}\mathrm{nd}(\tilde{\mathcal{X}}_\alpha)$ is supported below (and including) this line. Moreover, since the diagonal coming from $S_*^\alpha$ is bounded by assumption, we may find some sufficiently large constant $K$ such that $\mathbb{E}\mathrm{nd}(\tilde{\mathcal{X}}_\alpha)$ is supported above the line $y=(2+1/d)x-K$.
\end{remark}
\begin{lemma}
\emph{The tt-categories $\cat{K}_n = \DTM_{\tilde{\mathcal{X}}_{\rho^{n-1}}}(F(Q_{n}); \F_2)$ for $ 3\leq n < \infty$  satisfy Assumption \ref{assumption: computation m geq 2}}.
\end{lemma}
\begin{proof}
From Lemma \ref{lemma: function fields of real closed anisotropic quadrics} we know that $\K_l^{M}(F(Q_n))/2=0$ for $l \geq \mathrm{dim}(Q_n)=2^n-2$. Moreover, $\rho^{n}=0$ in $\K_n^{M}(F(Q_n)/2$, which means that $\rho \in \ker(\rho^{n-1} \cdot)$. From this, the required verification becomes immediate.
\end{proof}
For ease of notation, we shall denote the tensor unit $\tilde{\mathcal{X}}_{\alpha}$ of $\cat{L}_m$ by $\unit$. By the previous lemma, the computations of $\Spc(\cat{K}^c_n)$ for $ 3\leq n < \infty$  follow from the next result.
\begin{thm}\label{thm: case m geq 3}
 \emph{Under Assumption \ref{assumption: computation m geq 2}, the Balmer spectrum of $\cat{L}_m^c$ is homeomorphic to the following two-point Sierpiński space.}
 \[\begin{tikzcd}[ampersand replacement=\&,sep=small]
	0 \\
	{\langle \cone(r_{m-1})\rangle}
	\arrow[no head, from=1-1, to=2-1]
\end{tikzcd}\]
\end{thm}
\begin{proof}
Because $0$ is prime by Lemma \ref{lemma: chow weight structure and locality of cellular tt-categories}, then arguing as in the proof of Theorem \ref{thm: case m=1}, it suffices to establish the following.
\begin{enumerate}
    \item $\supp(\cone(r_{m-1}))=\{0\}$.
    \item $U(\cone(r_{m-1}))=\{\langle \cone(r_{m-1}) \rangle\}$.
\end{enumerate}
The first claim will follow from Proposition \ref{prop: case m geq 3 support of cone(r_m-1)}. For the second claim, we argue analogously to the proof of Theorem \ref{thm: case m=1} and invert $r_{m-1}$ by considering the smashing localisation corresponding to the following idempotent triangle,
\[\mathcal{D}_{r_{m-1}} \to \unit \to \tilde{\mathcal{D}}_{r_{m-1}} \to \mathcal{D}_{r_{m-1}}[1],\]
where 
\[\tilde{\mathcal{D}}_{r_{m-1}}\cong \hocolim_{n}\big(\unit \xrightarrow{r_{m-1}} \unit(-d_{m-1})[-2d_{m-1}-1]\xrightarrow{r_{m-1}(-d_{m-1})[-2d_{m-1}-1]}\dots).\] 
In particular, the tt-category
\[\cat{M}_m=\tilde{\mathcal{D}}_{r_{m-1}} \otimes \cat{L}_m\]
gets identified with the localisation of $\cat{L}_m$ at the localising tensor ideal generated by $\cone(r_{m-1})$. This reduces (2) to showing that $\Spc(\cat{M}_m^c)=\{0\}$. We will prove this in Proposition \ref{prop: case m geq 3 open of cone(r_m-1)}, where we will use an appropriate weight structure that we will define on $\cat{M}^c_m$ in Lemma \ref{lemma: weight structure on Mcm}, to show that the tt-ideal generated by any non-zero object is the whole tt-category $\cat{M}^c_m$.
\end{proof}

\begin{prop}\label{prop: case m geq 3 support of cone(r_m-1)}
\emph{For every non-zero object $A \in \cat{L}_m^c$, $\cone(r_{m-1}) \in \langle A \rangle.$}  
\end{prop}
\begin{proof}
Arguing as in Proposition \ref{prop: case m=2 support of cone(r_0)}, it suffices to show that for any object $Y \in (\cat{L}^c_m)^{\geq 1}$, any morphism $s:Y \to \unit$ is $\otimes$-nilpotent on $\cone(r_{m-1})$. We will do this by proving the stronger statement that there is a sufficiently large constant $K$ such that 
\[\Hom(\cone(r_{m-1})(i)[j], \cone(r_{m-1}))=0 \text{ for } j-2i\geq K.\]
This is sufficient since  $Y \in (\cat{L}^c_m)^{\geq 1}$, in particular for $n \geq K$, $Y^{\otimes n}$ is an extension of $\unit(i)[j]$ with $i$ and $j$ in the above range. Applying $\Hom(-, \cone(r_{m-1}))$ to the defining triangle of $\cone(r_{m-1})(i)[j]$, we observe that it is sufficient to check that 
\[ \text{(I)} \, \, \Hom(\unit(i-d_{m-1})[j-2d_{m-1}-1], \cone(r_{m-1}))=0, \quad \text{(II)} \, \, \Hom(\unit(i)[j+1],  \cone(r_{m-1}))=0 .\]
Moreover, applying $\Hom(\cone(r_{m-1})(i')[j'], -)$ to the defining triangle of $\cone(r_{m-1})$, and using the fact that multiplication by $r_{m-1}$ is injective in $\mathbb{E}\mathrm{nd}(\unit)$, we deduce that 
\[ \text{(III)} \, \, \Hom(\unit(i')[j'], \cone(r_{m-1}))\cong \Hom(\unit(i')[j'],\unit(-d_{m-1})[-2d_{m-1}-1])/(\im(r_{m-1})_*).\]
The term on the right can be computed by modding out $\mathbb{E}\mathrm{nd}(\unit)$ by the ideal generated by $r_{m-1}$. By Assumption \ref{assumption: computation m geq 2}, the resulting bi-graded ring is concentrated in a bounded region of $\Z^2$ (cf. Remark \ref{remark: support of the bigraded endomorphism ring of the unit for m geq 2}). The condition $j-2i \geq K$ implies that the bi-degrees of (I) and (II) both lie in the half plane defined by $y-2x \geq K-1$, hence it is clear that since (III) is supported in a bounded region, we may choose $K$ sufficiently large such that (I) and (II) vanish simultaneously.
\end{proof}

It remains to establish that the spectrum of $\cat{M}^c_m$--defined in the proof of Theorem \ref{thm: case m geq 3}--contains $0$ as its unique tt-prime. We begin by defining a weight structure on this tt-category.

First, observe that 
\[\mathbb{E}\mathrm{nd}(\tilde{\mathcal{D}}_{r_{m-1}})\cong\colim_n(\mathbb{E}\mathrm{nd}(\unit)),\]
where the transition maps are given by multiplication by $r_{m-1}$. In particular, the above ring is isomorphic to $\mathbb{E}\mathrm{nd}(\unit)[r_{m-1}^{-1}]$.
\begin{remark}\label{remark: boundedness of bigraded endomorphism ring in Mcm}
Visualising $\mathbb{E}\mathrm{nd}(\tilde{\mathcal{D}}_{r_{m-1}})$ in the lattice $\Z^2$, we observe that--as in Remark \ref{remark: support of the bigraded endomorphism ring of the unit for m geq 2}-- this is also supported below and including the line $y=(2+1/d)x$, with $d:=d_{m-1}$, and above a line $y=(2+1/d)x-K$ for $K$ a sufficiently large fixed constant.
\end{remark}

For ease of notation, we will denote the unit $\tilde{\mathcal{D}}_{r_{m-1}}$ of $\cat{M}_m$ by $\tunit$. 

\begin{lemma}\label{lemma: weight structure on Mcm}
\emph{The category $\cat{M}^c_m$ affords a unique bounded weight structure whose heart $\cat{H}$ is the additive subcategory of $\cat{M}^c_m$ generated by
\[\tunit\{i\}, \quad i=0, \dots, d-1,\]
where $d:=d_{m-1}$.}
\end{lemma}
\begin{proof}
From the computation of the bi-graded endomorphism ring of the unit, observe that 
\[\Hom(\tunit, \tunit\{n\})=0 \quad \text{for } -(d-1) \leq n < 0,\]
since the points $(n,2n)$ lie above the line $y=(2+1/d)x$ for $n$ in the above range, hence the corresponding groups vanish by Remark \ref{remark: boundedness of bigraded endomorphism ring in Mcm}. For $0\le n\le d-1$, the only possible non-zero groups $\Hom(\tunit,\tunit\{n\})$ occur in the bi-degrees obtained from the generators $r_i$, for $1\leq i\leq m-2$, after translating along the $(d)[2d+1]$-periodicity coming from $r_{m-1}^{-1}$. More explicitly,
\[\Hom(\tunit,\tunit\{n\})\cong
\begin{cases}
\F_2 & \text{if } n=0 \text{ or } n=d-d_i \text{ for some } i=1,\dots,m-2,\\
0 & \text{otherwise}.
\end{cases}\]
From the above we deduce that $\cat{H}$ is $\F_2$-linear, Hom-finite, and Krull-Schmidt, hence idempotent-complete \cite[Theorem 6.1]{shah2023krull}. Moreover, the $(d)[2d+1]$-periodicity of $\cat{M}_m$ implies that the thick closure of $\cat{H}$ is $\cat{M}^c_m$; in particular, by Theorem \ref{thereom: bondarko2010weight}, it suffices to verify the negativity condition. This reduces to checking that
\[\Hom(\tunit, \tunit(a)[2a+n])=0, \quad \text{for } -(d-1) \leq a \leq d-1, \text{ and } n \geq 1 \]
but this follows immediately from the fact that the corresponding points $(a,2a+n)$ lie above the line $y=(2+1/d)x$, hence the corresponding groups are zero by Remark \ref{remark: boundedness of bigraded endomorphism ring in Mcm}.
\end{proof}
We will call the weight structure defined on $\cat{M}^c_m$ the \emph{periodic Chow weight structure}. Observe that, in contrast to the usual Chow weight structure, the heart $\cat{H}$ is neither a tensor nor semisimple. The former fact is clear: for instance, $\tunit\{n\} \notin \cat{H}$ for $n \geq d$. Furthermore, $\cat{H}$ is not semisimple since--as we have seen in the proof of the previous lemma--there are non-trivial $\otimes$-nilpotent morphisms between the indecomposable generators.

We will use this weight structure to show that $\cat{M}^c_m$ contains enough $\otimes$-nilpotent morphisms: this will allow us to show that the tt-ideal
generated by any non-zero object in $\cat{M}^c_m$ contains the unit. We begin by proving the following technical lemma.
\begin{lemma}
\emph{For any $A \in \cat{M}^c_m$, there exists $i \in \{0, \dots, d-1\}$ such that $A(-i)[-2i]$ admits a weight complex supported on a segment $[s,t]$ where the last differential $\delta_{s+1}$ is zero.}
\end{lemma}
\begin{proof}
We may assume $A$ is non-zero. Since the periodic Chow weight structure is bounded, $A$ admits a weight complex $t(A)$ supported on some segment $[s,t]$:
\[\begin{tikzcd}[ampersand replacement=\&]
	{t(A):= \quad 0} \& {C_{t}} \& {C_{t-1}} \& \dots \& {C_{s+1}} \& {C_s} \& 0
	\arrow[from=1-1, to=1-2]
	\arrow["\delta_t", from=1-2, to=1-3]
	\arrow["", from=1-3, to=1-4]
	\arrow["", from=1-4, to=1-5]
	\arrow["\delta_{s+1}", from=1-5, to=1-6]
	\arrow[from=1-6, to=1-7]
\end{tikzcd}\]
where $C_s$ is non-zero, and each $C_k$, for $ s \leq k\leq t$ is a finite direct sum of objects $\tunit\{i\}$ with $i \in \{0,\dots, d-1\}$. First, if the last differential $\delta_{s+1}=0$, there is nothing to prove, so we may assume it is non-zero and proceed as follows. Consider the  shifted object $A' := A\{-i_{M}\}$ 
where $i_M$ is the largest index $i \in \{0, \dots,  d-1\}$ such that $\tunit\{i\}$ is a summand of $C_s$. From the $(d)[2d+1]$-periodicity, we observe that the lowest graded component of $A'$ lives in degree $s$ and it consists of a finite direct sum of copies of $\tunit$ only. This ensures that we may find a weight complex for $A'$ supported on $[s, t']$ where the lowest term is of the form 
\[C'_s = \bigoplus_a \tunit.\] 
If the new differential $\delta'_{s+1} = 0$, we are done. If not, recall from the proof of Lemma \ref{lemma: weight structure on Mcm} that 
\[\Hom(\tunit, \tunit\{j\})=0 \quad \text{for } -(d-1) \leq j < 0,\]
and hence, by duality,
\[\Hom(\tunit\{i\}, \tunit) = 0, \text{ for } 0 < i \leq d-1 .\] 
Thus, non-zero components of $\delta'_{s+1}: C'_{s+1} \to C'_s$ only arise from summands of $\tunit$ in $C'_{s+1}$. We may decompose 
\[C'_{s+1} = (\bigoplus_b \tunit) \oplus B,\] 
where $B$ does not contain any $\tunit$ as summands. The differential $\delta'_{s+1}$ is then represented by a matrix which vanishes on $B$, and acts as a matrix over $\mathrm{End}(\tunit) \cong \F_2$ on the summand $\bigoplus_b \tunit$. By Gaussian elimination, we may find automorphisms of the source and target to diagonalise this matrix as follows.
\[\left[ \begin{matrix} I_r & 0 \\ 0 & 0 \end{matrix} \right]\] 
This means that $\delta'_{s+1}$ can be identified with a morphism $1_W \oplus 0: C'_{s+1} \to C'_{s}$, where $W$ is a direct sum of some copies of $\tunit$. Arguing as in \cite[Proof of Prop 5.10]{vishik2024isotropic}, we may split off $W$ from the weight filtration, leaving a weight complex for $A'$ with a zero differential at position $s+1$.
\end{proof}

\begin{prop}\label{prop: case m geq 3 open of cone(r_m-1)}
\emph{For any non-zero object $B \in \cat{M}^c_m$, we have $\langle B \rangle =  \cat{M}^c_m$.}
\end{prop}
\begin{proof}
By the previous lemma, there exists a shift $A = B\{-i\}$ admitting a weight complex supported on a segment $[s,t]$ where the differential $\delta_{s+1}$ is zero, and the lowest non-zero graded component $A_s$ is a direct sum of copies of $\tunit[s]$. Consider the distinguished triangle
\[\begin{tikzcd}[ampersand replacement=\&, sep=small]
	{A_{>s}[-1]} \& {A_s} \& A \& {A_{>s}}
	\arrow["f",from=1-1, to=1-2]
	\arrow[from=1-2, to=1-3]
	\arrow[from=1-3, to=1-4]
\end{tikzcd}\]
and observe that to prove the proposition, it suffices to show that $f$ is $\otimes$-nilpotent. Since $\delta_{s+1}=0$, then $f$ factors through a morphism $g: A_{>s+1}[-1] \to A_s$. By rigidity, the $\otimes$-nilpotence of $g$ is equivalent to the $\otimes$-nilpotence of its adjoint 
\[\bar{g}: \tunit \to (A_{>s+1})^\vee[1] \otimes A_s=:E.\]
Since $A_{>s+1}$ is an extension of finitely many $\tunit(a)[b]$ with $0\leq a\leq d-1$ and $b-2a \geq s+2$, then $E$ is an extension of finitely many $\tunit(a)[b]$ with $-(d-1)\leq a \leq 0$ and $b-2a \leq -1$. In particular, this means that for any positive integer $n$, $E^{\otimes n}$ is an extension of finitely many $\tunit(x)[y]$ such that $(x,y)$ belong to the following region of $\Z^2$
\[L_n:\, -n(d-1)\leq x \leq 0, \quad y-2x \leq -n.\]
Since $\mathbb{E}\mathrm{nd}(\tunit)$ is concentrated in a strip bounded above by the line $l_1: y=(2+1/d)x$ and below by the line $l_2: y=(2+1/d)x-K$; then it is sufficient to show for sufficiently large $n$, the region $L_n$ is below the line $l_2$. This is because $E^{\otimes n}$ is an extension of $\tunit(x)[y]$ for $(x,y) \in L_n$, hence the above would imply that any morphism $\tunit \to  E^{\otimes n}$ is forced to be zero. We may choose $n > Kd$ and check that for $-n(d-1)\leq x \leq 0$, 
\[y-(2+1/d)x \leq y-2x + n(d-1)/d \leq -n + n(d-1)/d\leq -n/d < -K.\]
\end{proof}
\subsection{Isotropic case}
It remains to compute $\Spc(\cat{K}^c_{\infty})$. We begin by recording the computation of the bi-graded endomorphism ring of its unit. 
\begin{lemma}\label{lemma: bigraded endo ring of upsilon}
\emph{The bi-graded endomorphism ring of the unit $\mathbb{E}\mathrm{nd}(\Upsilon^{iso})$ in $\cat{K}_{\infty}$ is isomorphic to }
\[\F_2\big[\rho,\{r_i\}_{i \in \N_0} \big] \big /(\{r_{i}^2-r_{i+1}\rho\}_{i \in \N_0}).\]
\end{lemma}
\begin{proof}
Recall from Section \ref{section: real isotropic category of quadrics} that $\Upsilon^{iso} \cong \hocolim_n(\tilde{\mathcal{X}_n})$, in particular
\[\mathbb{E}\mathrm{nd}(\Upsilon^{iso}) \cong \colim_n \mathbb{E}\mathrm{nd}(\tilde{\mathcal{X}}_{n}).\]
The rest of the proof follows from the same argument in the proof of \cite[Theorem 3.7]{vishik2022isotropic} combined with the fact that when $F$ is real closed, multiplication by $\rho$ is injective in $\K^M_*(F)/2$ and 
\[\mathbb{E}\mathrm{nd}(\tilde{\mathcal{X}}_{n})=\F_2[ \rho, r_0, \dots, r_{n-1}]\big /(\{r_{i}^2-r_{i+1}\rho\}_{0 \leq i \leq n-2}).\]
\end{proof}
We begin by defining a weight structure on $\cat{K}_{\infty}^c$ that will allow us to deduce that this tt-category is local.
\begin{lemma}\label{lemma: weight structure on isotropic cat and locality}
\emph{The category $\cat{K}^c_\infty$ affords a unique bounded Chow weight structure whose heart $\cat{H}$  is the additive subcategory generated by the pure shifts of the unit,
\[\Upsilon^{iso}(i)[2i], \quad i \in \Z.\]
In particular, $\cat{K}^c_\infty$ is a local tt-category.  }
\end{lemma}
\begin{proof}
This follows from the same argument in the proof of Lemma \ref{lemma: chow weight structure and locality of cellular tt-categories}, together with the fact that the computation of the bi-graded endomorphism ring allows us to verify that 
\[\Hom(\Upsilon^{iso}, \Upsilon^{iso}(j)[2j+n])=0, \text{ for } j \in \Z \text{ and } n \geq 1.\]
The second claim follows from Remark \ref{remark: weight complex functors and local tt-categories} after observing that $\cat{H}$ is tensor and semisimple, and that $\cat{K}^c_\infty$ admits a stable symmetric monoidal $\infty$-categorical enhancement.
\end{proof}
\begin{thm}\label{thm: spectrum of isotropic case}
\emph{The Balmer spectrum of $\cat{K}^c_\infty$ consists of the following three-point space,}
\[\begin{tikzcd}[ampersand replacement=\&]
	0 \\
	{\langle \cone (\rho)^{iso} \rangle} \\
	{\langle \cone (r_0)^{iso} \rangle}
	\arrow[no head, from=1-1, to=2-1]
	\arrow[no head, from=2-1, to=3-1]
\end{tikzcd}\]
\emph{where the closed subsets are the specialisation-closed subsets.}
\end{thm}
In the above picture, we used the superscript  $iso$ to stress that these objects are considered in the isotropic category. However, for ease of notation, we shall drop them in this section since we will work only with the isotropic category.
\begin{proof}
We have already established in Lemma \ref{lemma: weight structure on isotropic cat and locality} that $0$ is prime. By decomposing $\Spc(\cat{K}^c_\infty)$ as 
\[\supp(\cone(\rho)) \sqcup U(\cone(\rho)),\]
we see that to describe the points of the spectrum, it suffices to show that
\begin{enumerate}
    \item $\supp(\cone(\rho))=\{0\}$,
    \item $U(\cone(\rho))=\{\langle \cone(\rho) \rangle, \langle \cone(r_0) \rangle\}$.
\end{enumerate}
Using an analogous strategy to the proof of Theorem \ref{thm: case m=1} and Theorem \ref{thm: case m geq 3}, (1) will follow from Proposition \ref{prop: support of cone(rho) in isotropic case}. For the second claim, we argue analogously to the proof of Theorem \ref{thm: case m=1} and invert $\rho$ by considering the smashing localisation corresponding to the following idempotent triangle,
\[\mathcal{D}_{\rho} \to \Upsilon^{iso} \to \tilde{\mathcal{D}}_{\rho} \to \mathcal{D}_{\rho}[1],\]
where 
\[\tilde{\mathcal{D}}_{\rho} \cong \hocolim_{n}\big(\Upsilon^{iso} \xrightarrow{\rho} \Upsilon^{iso}(1)[1]  \xrightarrow{\rho(1)[1]}\dots).\] 
In particular, the tt-category
\[ \cat{M}_\infty= \tilde{\mathcal{D}}_\rho \otimes \cat{K}_\infty\]
gets identified with the localisation of $\cat{K}_\infty$ at the localising tensor ideal generated by $\cone(\rho)$. This reduces (2) to computing $\Spc(\cat{M}^c_\infty)$. By defining an appropriate weight structure on $\cat{M}^c_\infty$ in Lemma \ref{lemma: weight structure on Mcinfty}, we will deduce that $0$ is a tt-prime of $\cat{M}^c_\infty$; in particular, $\langle \cone(\rho) \rangle$ is prime in $\cat{K}^c_\infty$. This will further reduce (2) to showing that 
\begin{enumerate}
\setcounter{enumi}{2}
    \item $\langle \tilde{\mathcal{D}}_\rho \otimes \cone(r_0) \rangle$ is the only non-zero prime tt-ideal of $\cat{M}^c_\infty$.
\end{enumerate}
Since the pull-back on spectra induced by the localisation
\[\Spc(\cat{M}^c_\infty) \to \Spc(\cat{K}^c_\infty)\]
is surjective on $U(\cone(\rho))$, and hence it covers $U(\cone(r_0)) \subseteq U(\cone(\rho))$ since $\cone(\rho) \in \langle \cone(r_0) \rangle$ by Proposition \ref{prop: support of cone(rho) in isotropic case}. In particular, assuming (3), we may deduce from  \cite[Proposition 2.10]{balmer2022three} that the prime corresponding to the pullback of 
\[ \langle\tilde{\mathcal{D}}_\rho \otimes \cone(r_0) \rangle\] 
is indeed generated by $\cone(r_0)$. From here, the description of the specialisation relations is straightforward. Moreover, it is clear that in this case the closed subsets are exactly those subsets which are closed under the specialisation relation. Thus, it remains to establish (3): to do this, it suffices to show that
\begin{enumerate}
\setcounter{enumi}{3}
    \item $\supp(\tilde{\mathcal{D}}_\rho \otimes \cone(r_0))=\{0\}$,
    \item $U(\tilde{\mathcal{D}}_\rho \otimes \cone(r_0))= \{\langle \tilde{\mathcal{D}}_\rho \otimes \cone(r_0) \rangle \}$.
\end{enumerate}
Claim (4) will follow from Proposition \ref{prop: support of cone(r_0) in Mcinfty}. Finally, to establish (5) we may formally invert $r_0$ by considering the smashing localisation of $\cat{M}_\infty$ corresponding to the idempotent triangle 
\[\mathcal{D}_{r_0} \to  \tilde{\mathcal{D}}_{\rho} \to \tilde{\mathcal{D}}_{r_0} \to \mathcal{D}_{r_0}[1],\]
where
\[\tilde{\mathcal{D}}_{r_0}=\hocolim_{n}\big(\tilde{\mathcal{D}}_\rho \xrightarrow{r_0} \tilde{\mathcal{D}}_\rho[-1] \xrightarrow{r_0[-1]}   \dots).\]
In particular, the tt-category
\[ \cat{N}_\infty= \tilde{\mathcal{D}}_{r_0} \otimes \cat{M}_\infty\]
gets identified with the localisation of $\cat{M}_\infty$ at the localising tensor ideal generated by $\tilde{\mathcal{D}}_\rho \otimes \cone(r_0)$. This reduces (5) to proving that $\Spc(\cat{N}_\infty^c)=\{0\}$. To show this, first observe that in $\cat{M}_\infty$,
\[\mathbb{E}\mathrm{nd}(\tilde{\mathcal{D}}_\rho)\cong \colim_n(\mathbb{E}\mathrm{nd}(\Upsilon^{iso})),\]
where the transition maps are given by multiplication by $\rho$. In particular, the above ring is isomorphic to 
\[\mathbb{E}\mathrm{nd}(\Upsilon^{iso})[\rho^{-1}].\]
Moreover, from Lemma \ref{lemma: bigraded endo ring of upsilon}, we have a description of the above ring and, by induction on the relations among its generators we obtain that for each $i \geq 1$, $r_{i}=\rho^{-(2^i-1)}r_0^{2^i}$. This tells us that  
\[\mathbb{E}\mathrm{nd}(\tilde{\mathcal{D}}_\rho)\cong \F_2[\rho, \rho^{-1}][r_0].\]
From this, we may deduce that since in  $\cat{N}_\infty$
\[\mathbb{E}\mathrm{nd}(\tilde{\mathcal{D}}_{r_0}) \cong \colim_{n} (\mathbb{E}\mathrm{nd}(\tilde{\mathcal{D}}_\rho)),\]
where the transition maps are given by multiplication by $r_0$, then
\[\mathbb{E}\mathrm{nd}(\tilde{\mathcal{D}}_{r_0}) \cong \F_2[\rho, \rho^{-1}, r_0, r_0^{-1}].\]
Because the thick closure of $\tilde{\mathcal{D}}_{r_0}$ is the whole category $\cat{N}_\infty^c$, then the description of the bi-graded endomorphism ring of the unit shows that every non-zero object of this category will be isomorphic to direct sums of $\tilde{\mathcal{D}}_{r_0}$. In particular, $\cat{N}_\infty^c$ is equivalent to the category of finite dimensional $\F_2$-vector spaces, hence it is clear that  $\Spc(\cat{N}_\infty^c)=\{0\}$.
\end{proof}

\begin{prop}\label{prop: support of cone(rho) in isotropic case}
\emph{For every non-zero object $A$ in $\cat{K}^c_\infty$, we have that $\cone(\rho) \in \langle A \rangle.$}
\end{prop}
\begin{proof}
Arguing as in Proposition \ref{prop: case m=2 support of cone(r_0)}, we know that using the weight structure defined in Lemma \ref{lemma: weight structure on isotropic cat and locality}, we may reduce the statement to checking for any object $Y \in (\cat{K}^c_\infty)^{\geq 1}$, any morphism $s:Y \to \Upsilon^{iso}$ is $\otimes$-nilpotent on $\cone(\rho)$. We will do this by proving the stronger statement that for any fixed $K \geq 0$, there is a sufficiently large integer $n$ such that 
\[\Hom(\cone(\rho)(i)[j], \cone(\rho))=0, \quad \text{ for }  j-2i \geq n \text{ and } -nK\leq i\leq nK.\]
This implies what we are after because we know that $Y$ is an extension of finitely many $\Upsilon^{iso}(i)[j]$ for $j-2i \geq 1$; in particular, we can find a sufficiently large constant $K$ such that $-K\leq i \leq K$. Hence, $Y^{\otimes n} \otimes \cone(\rho)(i)[j]$ is an extension of $\cone(\rho)(i)[j]$ for $j-2i \geq n$ and $-nK\leq i \leq nK$, from which the result becomes clear.

Applying $\Hom(-, \cone(\rho))$ to the defining triangle of $\cone(\rho)(i)[j]$, we observe that it is sufficient to check that 
\[\Hom(\Upsilon^{iso}(i+1)[j+1], \cone(\rho))=0, \quad \Hom(\Upsilon^{iso}(i)[j+1],  \cone(\rho))=0 .\]
Moreover, applying $\Hom(\cone(\rho)(i')[j'], -)$ to the defining triangle of $\cone(\rho)$, and using the fact that multiplication by $\rho$ is injective in $\mathbb{E}\mathrm{nd}(\Upsilon^{iso})$, we deduce that 
\[ \Hom(\Upsilon^{iso}(i')[j'], \cone(\rho))\cong \Hom(\Upsilon^{iso}(i')[j'],\Upsilon^{iso}(1)[1])/(\im(\rho_*)).\]
Combining everything together, we are reduced to check that, for any fixed $K \geq 0$, there is a sufficiently large integer $n$ such that 
\[\text{(I)} \, \, \Hom(\Upsilon^{iso}(i)[j], \Upsilon^{iso})\big /\im(\rho(-1)[-1]_*), \quad \text{(II)} \, \, \Hom(\Upsilon^{iso}(i-1)[j], \Upsilon^{iso})\big /\im(\rho(-1)[-1]_*),\]
both vanish for $j-2i \geq n$ and $-nK\leq i\leq nK$. By duality, these groups may be computed by modding out $\mathbb{E}\mathrm{nd}(\Upsilon^{iso})$ by the ideal generated by $\rho$. The resulting bi-graded ring is isomorphic to the exterior algebra $\Lambda_{\F_2}(r_i, i\in \N_0)$. Finally, we deduce that (I) and (II) vanish in the desired bounds exactly by the same argument as in the proof of \cite[Theorem 3.5]{sparks2025tensor}.
\end{proof}

We now move to consider the spectrum of the tt-category $\cat{M}^c_\infty$ defined in the proof of Theorem \ref{thm: spectrum of isotropic case}. To simplify notation, we shall denote by $\unit$ its $\otimes$-unit $\tilde{\mathcal{D}}_\rho$. Recall that we have already seen in the proof of Theorem \ref{thm: spectrum of isotropic case} that 
\[\mathbb{E}\mathrm{nd}(\unit)\cong \colim_n(\mathbb{E}\mathrm{nd}(\Upsilon^{iso}))\cong \F_2[\rho, \rho^{-1}][r_0].\]
We proceed by defining a weight structure on $\cat{M}^c_\infty$.
\begin{lemma}\label{lemma: weight structure on Mcinfty}
\emph{The category $\cat{M}^c_\infty$ affords a unique bounded weight structure whose heart $\cat{H}$ is the additive subcategory generated by the unit. In particular, $\cat{M}^c_\infty$ is a local tt-category.}
\end{lemma}
\begin{proof}
It is clear that $\cat{H}$ is $\F_2$-linear and semisimple, in particular it is idempotent complete. Moreover, the thick closure of $\cat{H}$ is the whole $\cat{M}^c_\infty$ by the $(1)[1]$-periodicity. Finally, by Theorem \ref{thereom: bondarko2010weight}, it suffices to verify that
\[\Hom(\unit, \unit[n])=0, \quad \text{for } n \geq 1, \]
which follows from the description of the bi-graded endomorphism ring of the unit.

The second claim follows as usual from Remark \ref{remark: weight complex functors and local tt-categories} after observing that $\cat{H}$ is tensor and semisimple, and that $\cat{M}^c_\infty$ admits a stable symmetric monoidal $\infty$-categorical enhancement.
\end{proof}
Using the weight structure defined above, we discuss the remaining result on which the proof of Theorem \ref{thm: spectrum of isotropic case} relied on.
\begin{prop}\label{prop: support of cone(r_0) in Mcinfty}
\emph{For every non-zero object $A$ in $\cat{M}^c_\infty$, we have that $\tilde{\mathcal{D}}_\rho \otimes \cone(r_0) \in \langle A \rangle.$}
\end{prop}
\begin{proof}
For ease of notation, we shall denote throughout the proof $u:=\tilde{\mathcal{D}}_\rho \otimes \cone(r_0)$. Arguing as in Proposition \ref{prop: case m=2 support of cone(r_0)}, we know that using the weight structure defined in the previous lemma, the proposition reduces to verifying that for any object $Y \in (\cat{M}^c_\infty)^{\geq 1}$, any morphism $s:Y \to \unit$ is $\otimes$-nilpotent on $u$.  We will do this by proving the stronger statement that for any such $Y$,
\[\Hom(Y \otimes u, u)=0.\]
This follows from the following claim
\[\Hom(u[l], u)=0, \quad \text{ for all } l \geq 1.\]
Note that it is clear that the claim implies what we are after since $Y$ is an extension of finitely many $\unit[l]$ for $l \geq 1$. To prove this, we may argue as in the proof of Proposition \ref{prop: case m=2 support of cone(r_0)} and consider the triangle defining $u$ in $\cat{M}^c_\infty$ with the appropriate induced long exact sequences to reduce the claim to the verification that for all $l \geq 1$,
\[ \Hom(\unit(i)[l-1], u)=0, \quad \, \, \Hom(\unit(i)[l+1],  u)=0.\]
 Moreover, multiplication by $r_0$ is injective in $\mathbb{E}\mathrm{nd}(\unit)$, hence the above reduces to verifying that
\[\text{(I)} \, \, \Hom(\unit[l], \unit)\big /\im((r_0[1])_*)=0, \quad \text{(II)} \, \, \Hom(\unit[l+2], \unit)\big /\im((r_0[1])_*)=0,\]
for all $l \geq 1$. By duality, we may compute these groups by modding out the bi-graded endomorphism ring of the unit by the ideal generated by $r_0$: the resulting ring is isomorphic to $\F_2[\rho, \rho^{-1}]$, from which the vanishing of (I) and (II) becomes immediate.
\end{proof}
\section{The prime tt-ideals}
Recall that from Theorem \ref{theorem: barthel2024surjectivity} and Theorem \ref{theorem: joinlty conservative geo functors} we know that the following map
\[ \bigsqcup_{n \in \N \cup \{\infty\}}\Spc(\pi_n): \bigsqcup_{n \in \N \cup \{\infty\}}\Spc(\cat{K}^c_n) \to \Spc(\DQMgm(F; \F_2)), \]
is surjective. We will now verify that this map is also injective, which will give us a full description of all the prime tt-ideals of $\Spc(\DQMgm(F; \F_2))$. We begin by showing that for each $n \in \N \cup \{\infty\}$, $\Spc(\pi_n)$ is itself injective. These verifications will be done by inspection, and we refer the reader to Table \ref{tab:geometric_functors} to recall the behaviour of these functors.

\subsection*{Injectivity of $\Spc(\pi_1)$}
Let 
\[\p_{1,0}=\pi_1^{-1}(0), \quad \p_{1,1}=\pi_1^{-1}(\langle \cone(\tau)|_{\overline{F}} \rangle).\] 
Note that $\p_{1,0} \subset \p_{1,1}$, and the inclusion is proper since $\cone(\tau) \in \p_{1,1}$ but $\cone(\tau) \not \in \p_{1,0}$.
\begin{remark}\label{remark: generators case j=1}
Note that $\langle \tilde{\M}_1 \rangle \subseteq \p_{1,0}$, and we shall later prove that this is in fact an equality. Moreover, $\langle \cone(\tau) \rangle = \p_{1,1}$ since this corresponds to the kernel of étale realisation restricted to $\DQMgm(F; \F_2)$ \cite[Example 2.1]{vishik2025balmer}.
\end{remark}

\subsection*{Injectivity of $\Spc(\pi_2)$}
Let 
\[\p_{2,0}=\pi_2^{-1}(0), \quad \p_{2,1}=\pi_2^{-1}(\langle \cone(r_0)|_{F(Q_2)} \rangle)=\pi_2^{-1}(\langle \tilde{\M}_1|_{F(Q_2)} \rangle).\] Note that $\p_{2,0} \subset \p_{2,1}$, and the inclusion is proper since $\tilde{\M}_1 \in \p_{2,1}$ but $\tilde{\M}_1 \not \in \p_{2,0}$.
\begin{remark} \label{remark: generators case j=2}
Note that $\langle \M_1,\tilde{\M}_2 \rangle \subseteq \p_{2,0}$ and $\langle \M_1,\tilde{\M}_1 \rangle \subseteq \p_{2,1}$, and we shall later see that these are in fact equalities.
\end{remark}

\subsection*{Injectivity of $\Spc(\pi_n)$ for $3 \leq n< \infty$}
Let
\[\p_{n,0}=\pi_n^{-1}(0), \quad \p_{n,1}=\pi_n^{-1}(\langle \cone(r_{n-2})|_{F(Q_n)} \rangle)=\pi_n^{-1}(\langle \tilde{\M}_{n-1}|_{F(Q_n)}) \rangle).\] 
Note that $\p_{n,0} \subset \p_{n,1}$, and the inclusion is proper since $\tilde{\M}_{n-1} \in \p_{n,1}$ but $\tilde{\M}_{n-1} \not \in \p_{n,0}$.
\begin{remark} \label{remark: generators 3 leq j le infty}
Note that $\langle \M_{n-1},\tilde{\M}_n \rangle \subseteq \p_{n,0}$ and $\langle \M_{n-1},\tilde{\M}_{n-1} \rangle \subseteq \p_{n,1}$, and we shall later see that these are in fact equalities.
\end{remark}

\subsection*{Injectivity of $\Spc(\pi_\infty)$} The injectivity follows immediately from the fact that $\pi_\infty$ is the restriction to the compact objects of a localisation. We may also see this directly via a similar strategy as in the previous cases: let 
\[\p_{\infty,0}=\pi_\infty^{-1}(0), \, \, \p_{\infty,1}=\pi_\infty^{-1}(\langle \cone(\rho)^{iso} \rangle ), \, \, \p_{\infty,2}=\pi_\infty^{-1}(\langle \cone(r_0)^{iso} \rangle )=\pi_\infty^{-1}(\langle \cone(\rho)^{iso},\cone(\tau)^{iso} \rangle).\] 
We have inclusions $\p_{\infty,0} \subset \p_{\infty,1} \subset \p_{\infty,2}$ and they are proper since $\cone(\rho) \not \in \p_{\infty,0}$, but $\cone(\rho) \in \p_{\infty,1}$; similarly, $\cone(\tau) \in \p_{\infty,2}$, but $\cone(\tau) \not \in \p_{\infty,1}$.

\begin{remark}\label{remark: generators j=infty}
Note that $\langle \M_n | n \in \N \rangle = \p_{\infty,0}$ because, as we have already seen, $\pi_\infty$ is the restriction to the compact objects of the smashing localisation of the corresponding localising tensor ideal. Moreover, $\langle \cone(\rho) \rangle \subseteq \p_{\infty,1}$ and $\langle \cone(\rho), \cone(\tau) \rangle \subseteq \p_{\infty,2}$, and we shall later see that these are indeed equalities.
\end{remark}

\subsection*{Injectivity of $\bigsqcup_{n \in \N \cup \{\infty\}}\Spc(\pi_n)$}
To verify the injectivity of $\bigsqcup_{n \in \N \cup \{\infty\}}\Spc(\pi_n)$, we need to prove that $\p_{i,j}=\p_{i', j'}$ if and only if $i=i'$ and $j=j'$. The fact that $\Spc(\pi_n)$ is injective for all $n \in \N \cup \{\infty\}$ implies that if $i=i'$, then $j=j'$;  so we must show that if $i \neq i'$, then $\p_{i,j}\neq\p_{i', j'}$.

First, note that 
\[\p_{1,0}, \p_{1,1} \neq \p_{i,j} \text{ for } i \neq 1.\]
This is because 
\[\pi_1(\M_1)= \T|_{\overline{F}} \oplus \T|_{\overline{F}},\] 
is not contained in any prime tt-ideal of $\cat{K}_1^c$, so $\M_1 \not \in  \p_{1,0}, \p_{1,1}$. However, $\M_1 \in \p_{i,j}$ when $i \neq 1$.

Next, for each $2 \leq k< \infty$, we can recursively verify the following.
\[\p_{k,0}, \p_{k,1} \neq \p_{i,j} \text{ for all }  k+1 \leq i \leq \infty.\]
This is because \[\pi_k(\M_k)= \tilde{\mathcal{X}}_{\rho^{k-1}} \oplus  \tilde{\mathcal{X}}_{\rho^{k-1}}\{d_{k-1}\}\] 
is not contained in any prime tt-ideal of $\cat{K}_k^c$, so $\M_k \not \in  \p_{k,0}, \p_{k,1}$. However, $\M_k \in \p_{i,j}$ for all $i$ and $j$ in the specified range.
\subsection{Generators and specialisation relations}
Our current knowledge gives us the following picture of the Balmer spectrum of $\DQMgm(F; \F_2)$.
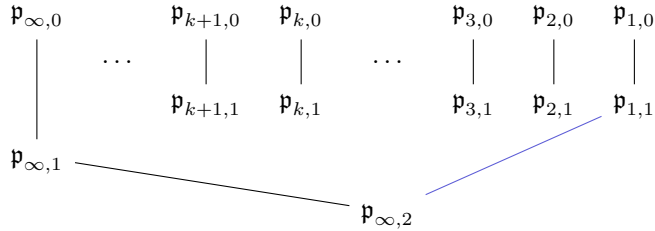
\[\begin{tikzcd}[ampersand replacement=\&, sep=tiny]
	{\p_{\infty,0}} \&\& {\p_{k+1,0}} \& {\p_{k,0}} \&\& {\p_{3,0}} \& {\p_{2,0}} \& {\p_{1,0}} \\
	\& \dots \&\&\& \dots \\
	\&\& {\p_{k+1,1}} \& {\p_{k,1}} \&\& {\p_{3,1}} \& {\p_{2,1}} \& {\p_{1,1}} \\
	{\p_{\infty,1}} \\
	\&\&\&\& {\p_{\infty,2}}
	\arrow[no head, from=1-1, to=4-1]
	\arrow[no head, from=1-3, to=3-3]
	\arrow[no head, from=1-4, to=3-4]
	\arrow[no head, from=1-6, to=3-6]
	\arrow[no head, from=1-7, to=3-7]
	\arrow[no head, from=1-8, to=3-8]
	\arrow[no head, from=4-1, to=5-5]
	\arrow[color={rgb,255:red,92;green,92;blue,214}, no head, from=5-5, to=3-8]
\end{tikzcd}\]
The black specialisations are clear, and the blue specialisation follows from the fact that by Remark \ref{remark: generators case j=1} and Remark \ref{remark: generators j=infty}
\[\p_{1,1}=\langle \cone(\tau) \rangle \subset \p_{\infty,2}.\]
To understand all the specialisation relations, we will first describe the generators of all the prime tt-ideals. Recall that, by Remark \ref{remark: generators case j=1} and Remark \ref{remark: generators j=infty}, we currently know that 
\[\p_{1,1}=\langle \cone(\tau) \rangle, \quad \p_{\infty,0}=\langle \M_n | n \in \N \rangle.\]
We can also easily find generators for $\p_{\infty,1}$ and $\p_{\infty,2}$. 
\begin{lemma}
\emph{The prime tt-ideals $\p_{\infty, 1}$ and $\p_{\infty,2}$ are generated by}
\[\p_{\infty, 1}= \langle \cone(\rho)\rangle, \quad \p_{\infty, 2}= \langle \cone(\rho), \cone(\tau)\rangle.  \]
\end{lemma}
\begin{proof}
Since $\DQMgm(F; \F_2)$ is a rigid tt-category, in particular any tt-ideal is radical. This means that for any set of objects $S$, the tt-ideal generated by $S$ is equal to the intersection of all primes containing $S$.
\begin{itemize}
    \item Let $S=\{ \cone(\rho), \cone(\tau)\}$. The only prime containing $S$ is $\p_{\infty, 2}$, from which the claim follows.
    \item Let $S=\{\cone(\rho)\}$. The primes containing $S$ are $\p_{\infty,1}$ and $\p_{\infty,2}$. Since $\p_{\infty, 1} \subset \p_{\infty, 2}$, their intersection is $\p_{\infty,1}$. This proves the claim.
\end{itemize}
\end{proof}
\subsection{Generators of $\p_{j,0}$ for $1\leq j <  \infty$}
We shall begin with the case $j=1$.

\begin{cor}{(\cite[Proposition 2.18]{vishik2025balmer})}
\emph{The prime tt-ideal $\p_{1,0}$ is generated by}
\[\p_{1,0}=\langle \tilde{\M}_1 \rangle.\]
\end{cor}
\begin{proof}
As discussed in Remark \ref{remark: generators case j=1}, the inclusion $ \langle \tilde{\M}_1 \rangle \subseteq \p_{1,0}$ is clear. For the other inclusion, suppose that $U \in \p_{1,0}$, then $\mathcal{X}_1 \otimes U\cong 0$. This implies that $\tilde{\mathcal{X}}_1 \otimes U \cong U$. Consider the following composition
\[f:U \cong \tilde{\mathcal{X}}_1 \otimes U \xrightarrow{r_0 \otimes id_U}\tilde{\mathcal{X}}_1 \otimes U[-1] \cong U[-1].\]
Since $U$ is a compact object which disappears in étale realisation, then by \cite[Proposition 2.18]{vishik2025balmer}, $f$ is $\otimes$-nilpotent. Moreover, $\cone(f) \cong U \otimes \cone(r_0)$, which implies that for $n$ sufficiently large, $U^{\otimes n}$ is (up to shift) a summand of $\cone(f^{\otimes n})$. Finally, since every tt-ideal is radical, we deduce that
\[U\in \langle \cone(f^{\otimes n}) \rangle \subseteq \langle \cone(f) \rangle \subseteq \langle \cone(r_0) \rangle=\langle \tilde{\M}_1 \rangle. \]
\end{proof}
In the following proposition, we describe the generators of the remaining primes in the top layer of the spectrum.

\begin{prop}\label{prop: generators of closed points}
\emph{For $2 \leq j < \infty$, the prime tt-ideals $\p_{j,0}$ are generated by}
\[\p_{j,0} = \langle \M_{j-1}, \tilde{\M}_j \rangle.\]
\end{prop}
\begin{proof}
By Remark \ref{remark: generators case j=2} and Remark \ref{remark: generators 3 leq j le infty}, the inclusion $\langle \M_{j-1}, \tilde{\M}_j \rangle \subseteq \p_{j,0}$ is clear. For the other inclusion, if $A \in \p_{j,0}$, then we know that $\tilde{\mathcal{X}}_{j-1} \otimes \mathcal{X}_j \otimes A \cong 0$, which implies that $\mathcal{X}_{j-1} \otimes A \cong \mathcal{X}_{j} \otimes A$, and $\tilde{\mathcal{X}}_{j-1} \otimes A \cong \tilde{\mathcal{X}}_{j} \otimes A$.  Since  $\mathcal{X}_{j-1} \otimes A \in \mathrm{Loc}_{\otimes}(\M_{j-1})$, then from the following triangle
\[\begin{tikzcd}[ampersand replacement=\&,sep=small]
	{\mathcal{X}_{j} \otimes A} \& A \& {\tilde{\mathcal{X}}_{j} \otimes A} \& {\mathcal{X}_{j} \otimes A[1]} \\
	{\mathcal{X}_{j-1} \otimes A} \&\& {\tilde{\mathcal{X}}_{j-1} \otimes A}
	\arrow[from=1-1, to=1-2]
	\arrow["\cong"{marking, allow upside down}, draw=none, from=1-1, to=2-1]
	\arrow[from=1-2, to=1-3]
	\arrow[from=1-3, to=1-4]
	\arrow["\cong"{marking, allow upside down}, draw=none, from=1-3, to=2-3]
\end{tikzcd}\]
we deduce that it is sufficient to show that $\tilde{\mathcal{X}}_{j} \otimes A \in \mathrm{Loc}_{\otimes}(\tilde{\M}_{j})$. Recall that $\tilde{\M}_{j}$ is isomorphic up to shift to $\cone( \tilde{\mathcal{X}}_{j} \xrightarrow{r_{j-1}} \tilde{\mathcal{X}}_{j}(-d_{j-1})[-2d_{j-1}-1])$; in particular, it suffices to show that $r_{j-1}$ is $\otimes$-nilpotent on $A$. Under the equivalences discussed above, we will do this by proving the stronger statement that there is a sufficiently large positive integer $K$ such that 
\[\Hom(\tilde{\mathcal{X}}_{j-1} \otimes A, \tilde{\mathcal{X}}_{j-1} \otimes A(-Kd_{j-1})[-2Kd_{j-1}-K] )=0.\]
Moreover, taking $C= A^{\vee} \otimes A$, the above is equivalent by duality to showing that 
\[\Hom(\tilde{\mathcal{X}}_{j-1}, \tilde{\mathcal{X}}_{j-1} \otimes C(-Kd_{j-1})[-2Kd_{j-1}-K] )=0\]
for $K$ sufficiently large. We proceed by recalling that the Chow weight structure on $\DMgm(F; \F_2)$ restricts to $\DQMgm(F; \F_2)$ \cite{bachmann2017invertibility}, and its heart consists of direct sums of pure Tate motives and pure Rost motives. Consequently, since $C \in \DQMgm(F; \F_2)^{\geq m}$ for some integer $m$, then $\tilde{{\mathcal{X}}}_{j-1} \otimes C$ is an extension of finitely many objects of the form $\tilde{{\mathcal{X}}}_{j-1}(x)[y]$ and $\tilde{{\mathcal{X}}}_{j-1} \otimes \M_l(x)[y]$ for $l \geq j$ and $y-2x \geq m$; this means that it suffices to show that for each shift $(x)[y]$ as above there is some sufficiently large $k$ such that 
\[(1) \quad \Hom(\tilde{\mathcal{X}}_{j-1}, \tilde{\mathcal{X}}_{j-1}(x-kd_{j-1})[y-2kd_{j-1}-k] )=0\]
and
\[(2) \quad \Hom(\tilde{\mathcal{X}}_{j-1}, \tilde{\mathcal{X}}_{j-1}\otimes \M_l (x-kd_{j-1})[y-2kd_{j-1}-k] )=0.\]
\textbf{Case} (1). If $j-1=1$, we know that the group vanishes as soon as $(x-kd_{j-1})$ becomes negative. For $j-1 \geq 2$, we proceed by a slope argument. We know that  $\mathbb{E}nd(\tilde{\mathcal{X}}_{j-1})$ is supported in a region of $\Z^2$ which is bounded by a line of slope $ 2 + 1/d_{j-2}$. However, as $k$ increases, the bi-degree of the target moves along a trajectory of slope $2 + 1/d_{j-1}$. Since $d_{j-1} > d_{j-2}$, then the target bi-degree will exit the region where $\mathbb{E}nd(\tilde{\mathcal{X}}_{j-1})$ is supported, ensuring that the group vanishes for sufficiently large $k$.

\textbf{Case} (2). It suffices to consider the vanishing of
\[(2') \quad \Hom(\T, \tilde{\mathcal{X}}_{j-1}\otimes \M_l (x-kd_{j-1})[y-2kd_{j-1}-k] ).\]
Tensoring the triangle $\mathcal{X}_{j-1} \to \T \to \tilde{\mathcal{X}}_{j-1} \to \mathcal{X}_{j-1}[1]$ by $\M_l$, and considering the long exact sequence induced by $\Hom(\T,-)$, it suffices to study the vanishing of the following groups.
\begin{enumerate}
    \item[(I)] $\Hom(\T, \M_l(x-kd_{j-1})[y-2kd_{j-1}-k])$.
    
    This vanishes for large $k$ because $(x-kd_{j-1})$ becomes negative.
    \item[(II)] $\Hom(\T, \mathcal{X}_{j-1} \otimes \M_l(x-kd_{j-1})[y-2kd_{j-1}-k+1])$. 
    
    Since $l \geq j$, we have $\mathcal{X}_{j-1} \otimes \mathcal{X}_l \cong \mathcal{X}_{j-1}$. Consequently, tensoring the defining triangle of $\M_l$ with $\mathcal{X}_{j-1}$ implies that $\mathcal{X}_{j-1} \otimes \M_l$ is an extension of 
    \[\mathcal{X}_{j-1}(d_{l-1}+\alpha_k)[2d_{l-1}+\beta_k] \, \text{ and  } \, \mathcal{X}_{j-1}(\alpha_k)[\beta_k].\]

    where $\alpha_k=x-kd_{j-1}$ and $\beta_k=y-2kd_{j-1}-k+1$.
    By shifting appropriately the idempotent triangle for $\mathcal{X}_{j-1}$, and considering the long exact sequence induced by $\Hom(\T,-)$, we may further reduce the vanishing of (II) to the vanishing of the following groups.
    \begin{itemize}
        \item $\Hom(\T, \T(d_{l-1}+\alpha_k)[2d_{l-1}+\beta_k])$ and $\Hom(\T, \T(\alpha_k)[\beta_k])$.

        These both vanish for large $k$ since $d_{l-1}+\alpha_k$ (and hence $\alpha_k$) becomes negative.
    \item $\Hom(\T,\tilde{\mathcal{X}}_{j-1}(d_{l-1}+\alpha_k)[2d_{l-1}+\beta_k +1])$ and $\Hom(\T,\tilde{\mathcal{X}}_{j-1}(\alpha_k)[\beta_k +1])$.

    These both vanish by the same arguments in Case (1).
    \end{itemize}
\end{enumerate}
This concludes the proof.
\end{proof}
Using the generators obtained for the points in the top layer of the spectrum, 
we may update the picture of $\Spc(\DQMgm(F; \F_2))$,
 \[\begin{tikzcd}[ampersand replacement=\&, sep=tiny]
	{\mathfrak{p}_{\infty,0}} \&\& {\mathfrak{p}_{k+1,0}} \& {\mathfrak{p}_{k,0}} \& {} \& {\mathfrak{p}_{3,0}} \& {\mathfrak{p}_{2,0}} \& {\mathfrak{p}_{1,0}} \\
	\& \dots \&\&\& \dots \\
	\& {} \& {\mathfrak{p}_{k+1,1}} \& {\mathfrak{p}_{k,1}} \& {} \& {\mathfrak{p}_{3,1}} \& {\mathfrak{p}_{2,1}} \& {\mathfrak{p}_{1,1}} \\
	{\mathfrak{p}_{\infty,1}} \\
	\&\&\&\& {\mathfrak{p}_{\infty,2}}
	\arrow[no head, from=1-1, to=4-1]
	\arrow[color={rgb,255:red,255;green,51;blue,75}, between={0}{0.4}, no head, from=1-3, to=3-2]
	\arrow[no head, from=1-3, to=3-3]
	\arrow[color={rgb,255:red,255;green,51;blue,75}, no head, from=1-4, to=3-3]
	\arrow[no head, from=1-4, to=3-4]
	\arrow[color={rgb,255:red,255;green,51;blue,75}, between={0}{0.4}, no head, from=1-6, to=3-5]
	\arrow[no head, from=1-6, to=3-6]
	\arrow[color={rgb,255:red,255;green,51;blue,75}, no head, from=1-7, to=3-6]
	\arrow[no head, from=1-7, to=3-7]
	\arrow[color={rgb,255:red,255;green,51;blue,75}, no head, from=1-8, to=3-7]
	\arrow[no head, from=1-8, to=3-8]
	\arrow[color={rgb,255:red,255;green,51;blue,75}, between={0}{0.4}, no head, from=3-4, to=1-5]
	\arrow[shift right=2, no head, from=4-1, to=5-5]
	\arrow[shift right=2, no head, from=5-5, to=3-8]
\end{tikzcd}\]
where the red lines denote the specialisation relations that can be easily deduced by inspection.
\subsection{Generators of $\p_{j,1}$ for $2 \leq j <  \infty$}
For $2 \leq j < \infty$, we shall denote
\[\tilde{\mathcal{D}}_{r_{j-2}}^{F(Q_j)}=\hocolim_{n \in \N}\big(\tilde{\mathcal{X}}_{\rho^{j-1}} \xrightarrow{r_{j-2}} \tilde{\mathcal{X}}_{\rho^{j-1}}(-d_{j-2})[-2d_{j-2}-1] \xrightarrow{r_{j-2}(-d_{j-2})[-2d_{j-2}-1]} \dots)\]
taken in $\DQM(F(Q_j); \F_2)$. Recall from Section \ref{section: Spc of Kc2} and Section \ref{section: Spc of Kcm} that for $2 \leq j < \infty$ 
\[\p_{j,1}=\ker(\tilde{\mathcal{D}}_{r_{j-2}}^{F(Q_j)} \otimes (\tilde{\mathcal{X}}_{j-1} \otimes -)|_{F(Q_j)}) \cap \DQMgm(F; \F_2).\]
Now, we may define in $\DQM(F; \F_2)$ 
\[\tilde{\mathcal{D}}_{r_{j-2}}=\hocolim_{n \in \N}\big(\tilde{\mathcal{X}}_{{j-1}}\xrightarrow{r_{j-2}} \tilde{\mathcal{X}}_{{j-1}}(-d_{j-2})[-2d_{j-2}-1]  \xrightarrow{r_{j-2}(-d_{j-2})[-2d_{j-2}-1]} \dots),\]
and it is clear that under base-change $\DQM(F; \F_2) \to \DQM(F(Q_j); \F_2)$, $\tilde{\mathcal{D}}_{r_{j-2}} \mapsto \tilde{\mathcal{D}}_{r_{j-2}}^{F(Q_j)}$. In particular, this gives us that 
\[\p_{j,1}=\ker(\tilde{\mathcal{D}}_{r_{j-2}} \otimes \mathcal{X}_{j} \otimes -) \cap \DQMgm(F; \F_2).\]
\begin{prop}
\emph{For $2\leq j <  \infty$, the prime tt-ideal $\p_{j,1}$ is generated by}
\[\p_{j,1}=\langle \M_{j-1}, \tilde{\M}_{j-1} \rangle\]
\end{prop}
\begin{proof}
As discussed in Remark \ref{remark: generators case j=2} and Remark \ref{remark: generators 3 leq j le infty}, the inclusion $\langle \M_{j-1}, \tilde{\M}_{j-1} \rangle \subseteq \p_{j,1}$ is clear. For the other inclusion, suppose that $A \in \p_{j,1}$, we will show that $\tilde{\mathcal{X}}_{j-1} \otimes A \in \mathrm{Loc}_{\otimes}(\tilde{\M}_{j-1})$; this will be sufficient because we already know that $\mathcal{X}_{j-1} \otimes A \in \mathrm{Loc}_{\otimes}(\M_{j-1})$, and the rest would follow from the familiar triangle 
\[\mathcal{X}_{j-1} \otimes A \to A \to \tilde{\mathcal{X}}_{j-1} \otimes A \to \mathcal{X}_{j-1} \otimes A[1]. \]
To prove the claim,  we first use the fact that by assumption $\tilde{\mathcal{D}}_{r_{j-2}} \otimes \mathcal{X}_{j} \otimes A\cong 0$; using the idempotent triangle for $\mathcal{X}_j$, this implies that $\tilde{\mathcal{D}}_{r_{j-2}} \otimes A \cong \tilde{\mathcal{D}}_{r_{j-2}} \otimes \tilde{\mathcal{X}}_j  \otimes A$. Next, we have a triangle
\[ \mathcal{D}_{r_{j-2}} \to \tilde{\mathcal{X}}_{j-1} \to \tilde{\mathcal{D}}_{r_{j-2}} \to \mathcal{D}_{r_{j-2}}[1],\]
where one can check that $\mathcal{D}_{r_{j-2}} \in \mathrm{Loc}_{\otimes}(\tilde{\M}_{j-1})$. Tensoring the above triangle by $A$, we see that to prove the claim it is sufficient to show that 
\[\tilde{\mathcal{D}}_{r_{j-2}} \otimes A \cong \tilde{\mathcal{D}}_{r_{j-2}} \otimes \tilde{\mathcal{X}}_j  \otimes A \in \mathrm{Loc}_{\otimes}(\tilde{\M}_{j}) \subseteq \mathrm{Loc}_{\otimes}(\tilde{\M}_{j-1}).\] 
Consider $r_{j-1}: \tilde{\mathcal{X}}_j  \to \tilde{\mathcal{X}}_j(-d_{j-1})[-2d_{j-1}-1]$, and note that since $\mathrm{Loc}_{\otimes}(\tilde{\M}_{j})=\mathrm{Loc}_{\otimes}(\cone(r_{j-1}))$, then it suffices to show that $r_{j-1}$ is $\otimes$-nilpotent on $\tilde{\mathcal{D}}_{r_{j-2}} \otimes A$. Under the equivalences discussed above, we will do this by proving the stronger statement that there is a sufficiently large $K$ such that 
\[\Hom(\tilde{\mathcal{D}}_{r_{j-2}} \otimes A, \tilde{\mathcal{D}}_{r_{j-2}} \otimes A(-Kd_{j-1})[-2Kd_{j-1}-K] )=0.\]
Moreover, taking $C= A^{\vee} \otimes A$, the above is equivalent by duality to showing that 
\[\Hom(\tilde{\mathcal{D}}_{r_{j-2}}, \tilde{\mathcal{D}}_{r_{j-2}} \otimes C(-Kd_{j-1})[-2Kd_{j-1}-K] )=0.\]
for $K$ sufficiently large. Now, arguing as in Proposition \ref{prop: generators of closed points}, we know that $\tilde{\mathcal{D}}_{r_{j-2}} \otimes C$ is an extension of finitely many objects of the form $\tilde{\mathcal{D}}_{r_{j-2}}(x)[y]$ and $\tilde{\mathcal{D}}_{r_{j-2}} \otimes \M_l(x)[y]$ for $l \geq j$ and $y-2x \geq m$, where $m$ is some fixed integer; this means that it suffices to show that for each shift $(x)[y]$ as above there is some sufficiently large $k$ such that
\[(1) \quad \Hom(\tilde{\mathcal{D}}_{r_{j-2}}, \tilde{\mathcal{D}}_{r_{j-2}}(x-kd_{j-1})[y-2kd_{j-1}-k] )=0\]
and
\[(2) \quad \Hom(\tilde{\mathcal{D}}_{r_{j-2}}, \tilde{\mathcal{D}}_{r_{j-2}}\otimes \M_l (x-kd_{j-1})[y-2kd_{j-1}-k] )=0.\]
To study the vanishing of these groups, recall that we may compute \[\mathbb{E}nd(\tilde{\mathcal{D}}_{r_{j-2}})=\colim_{n} \mathbb{E}nd(\tilde{\mathcal{X}}_{{j-1}}),\]
where the transition maps are given by multiplication by $r_{j-2}$, i.e.
\[\mathbb{E}nd(\tilde{\mathcal{D}}_{r_{j-2}})= \mathbb{E}nd(\tilde{\mathcal{X}}_{{j-1}})[r_{j-2}^{-1}].\]
In particular, the vanishing of $(1)$ follows from essentially the same verifications as in the proof of Proposition \ref{prop: generators of closed points}. For $(2)$, note that it suffices to study the vanishing of
\[ (2') \quad \Hom(\T, \tilde{\mathcal{D}}_{r_{j-2}}\otimes \M_l (x-kd_{j-1})[y-2kd_{j-1}-k]),\]
and the latter can be rewritten as the following colimit
\[\colim_{n} \Hom(\T, \tilde{\mathcal{X}}_{j-1}\otimes \M_l (x-kd_{j-1}-nd_{j-2})[y-2kd_{j-1}-k -2nd_{j-2}-n])).\]
By the analogous verification in the proof of Proposition \ref{prop: generators of closed points}, we can always find some sufficiently large $k$ such that the groups in the above colimit are (eventually) zero. 
\end{proof}
Using the description of the generators of all the tt-primes, we may update the previous picture of $\Spc(\DQMgm(F; \F_2))$.
\[\begin{tikzcd}[ampersand replacement=\&,sep=tiny]
	{\mathfrak{p}_{\infty,0}} \&\& {\mathfrak{p}_{k+1,0}} \& {\mathfrak{p}_{k,0}} \& {} \& {\mathfrak{p}_{3,0}} \& {\mathfrak{p}_{2,0}} \& {\mathfrak{p}_{1,0}} \\
	\& \dots \&\&\& \dots \\
	\& {} \& {\mathfrak{p}_{k+1,1}} \& {\mathfrak{p}_{k,1}} \& {} \& {\mathfrak{p}_{3,1}} \& {\mathfrak{p}_{2,1}} \& {\mathfrak{p}_{1,1}} \\
	{\mathfrak{p}_{\infty,1}} \\
	\&\&\&\& {\mathfrak{p}_{\infty,2}}
	\arrow[no head, from=1-1, to=4-1]
	\arrow[between={0}{0.4}, no head, from=1-3, to=3-2]
	\arrow[no head, from=1-3, to=3-3]
	\arrow[no head, from=1-4, to=3-3]
	\arrow[no head, from=1-4, to=3-4]
	\arrow[between={0}{0.4}, no head, from=1-6, to=3-5]
	\arrow[no head, from=1-6, to=3-6]
	\arrow[no head, from=1-7, to=3-6]
	\arrow[no head, from=1-7, to=3-7]
	\arrow[no head, from=1-8, to=3-7]
	\arrow[no head, from=1-8, to=3-8]
	\arrow["\dots"{description}, color={rgb,255:red,97;green,209;blue,108}, no head, from=3-2, to=5-5]
	\arrow[color={rgb,255:red,97;green,209;blue,108}, no head, from=3-3, to=5-5]
	\arrow[between={0}{0.4}, no head, from=3-4, to=1-5]
	\arrow[color={rgb,255:red,97;green,209;blue,108}, no head, from=3-4, to=5-5]
	\arrow[color={rgb,255:red,97;green,209;blue,108}, no head, from=3-6, to=5-5]
	\arrow[color={rgb,255:red,97;green,209;blue,108}, no head, from=3-7, to=5-5]
	\arrow[shift right=2, no head, from=4-1, to=5-5]
	\arrow["\dots"{description, pos=0.3}, color={rgb,255:red,97;green,209;blue,108}, between={0}{0.6}, no head, from=5-5, to=2-5]
	\arrow[shift right=2, no head, from=5-5, to=3-8]
\end{tikzcd}\]
Moreover, by inspection it becomes now straightforward to verify that there are no other specialisation relations among these points.
\section{The topology}
At this stage, we have a description of  $\Spc(\DQMgm(F; \F_2))$ as a set, we know the generators of all the prime tt-ideals and we know the specialisation relations among them. We now move to describe the topology. In general, $\Spc(\cat{K})$ has a basis of closed subsets given by
\[\{\supp(A) | A \in \cat{K}\}\]
and their complements $U(A)$ are quasi-compact. We will explicitly describe this basis by proving the following.
\begin{thm}\label{thm: basis of spectrum}
\emph{A basis of closed subsets of $\Spc(\DQMgm(F; \F_2)$ is given by specialisation closed subsets $Z$ such that
\begin{enumerate}
    \item if $\p_{\infty,0} \in Z$, then $Z$ is cofinite;
    \item if $\p_{\infty,0} \not \in Z$, then $Z$ is finite.
\end{enumerate}}
\end{thm}
\begin{table}[htbp]
    \centering
    \setlength{\tabcolsep}{3em}
\begin{tabular}{@{}cl@{}}
        \toprule
        \textbf{Support} & \textbf{Picture} \\
        \midrule      
1. $\supp(\cone(\rho))$ & $\begin{tikzcd}[ampersand replacement=\&,sep=small]
	{\mathfrak{p}_{\infty,0}} \&\& {\mathfrak{p}_{k+2,0}} \& {\mathfrak{p}_{k+1,0}} \& {\mathfrak{p}_{k,0}} \& {} \& {\mathfrak{p}_{2,0}} \& {\mathfrak{p}_{1,0}} \\
	\& \dots \&\&\&\& \dots \\
	\& {} \& {\mathfrak{p}_{k+2,1}} \& {\mathfrak{p}_{k+1,1}} \& {\mathfrak{p}_{k,1}} \& {} \& {\mathfrak{p}_{2,1}} \& {\mathfrak{p}_{1,1}}
	\arrow[between={0}{0.4}, no head, from=1-3, to=3-2]
	\arrow[no head, from=1-3, to=3-3]
	\arrow[no head, from=1-4, to=3-3]
	\arrow[no head, from=1-4, to=3-4]
	\arrow[no head, from=1-5, to=3-5]
	\arrow[between={0.6}{1}, no head, from=1-6, to=3-5]
	\arrow[between={0}{0.4}, no head, from=1-7, to=3-6]
	\arrow[no head, from=1-7, to=3-7]
	\arrow[no head, from=1-8, to=3-7]
	\arrow[no head, from=1-8, to=3-8]
	\arrow[no head, from=3-4, to=1-5]
\end{tikzcd}$ \\[8ex]
        
2. $\supp(\cone(\tau))$ & $\begin{tikzcd}[ampersand replacement=\&,sep=small]
	{\mathfrak{p}_{\infty,0}} \&\& {\mathfrak{p}_{k+2,0}} \& {\mathfrak{p}_{k+1,0}} \& {\mathfrak{p}_{k,0}} \& {} \& {\mathfrak{p}_{2,0}} \& {\mathfrak{p}_{1,0}} \\
	\& \dots \&\&\&\& \dots \\
	{\mathfrak{p}_{\infty,1}} \& {} \& {\mathfrak{p}_{k+2,1}} \& {\mathfrak{p}_{k+1,1}} \& {\mathfrak{p}_{k,1}} \& {} \& {\mathfrak{p}_{2,1}}
	\arrow[no head, from=1-1, to=3-1]
	\arrow[between={0}{0.4}, no head, from=1-3, to=3-2]
	\arrow[no head, from=1-3, to=3-3]
	\arrow[no head, from=1-4, to=3-3]
	\arrow[no head, from=1-4, to=3-4]
	\arrow[no head, from=1-5, to=3-5]
	\arrow[between={0.6}{1}, no head, from=1-6, to=3-5]
	\arrow[between={0}{0.4}, no head, from=1-7, to=3-6]
	\arrow[no head, from=1-7, to=3-7]
	\arrow[no head, from=1-8, to=3-7]
	\arrow[no head, from=3-4, to=1-5]
\end{tikzcd}$ \\[8ex]

3. $\supp(\tilde{\M}_k)$ & $\begin{tikzcd}[ampersand replacement=\&,sep=small]
	{\mathfrak{p}_{\infty,0}} \&\& {\mathfrak{p}_{k+2,0}} \& {\mathfrak{p}_{k+1,0}} \\
	\& \dots \\
	{\mathfrak{p}_{\infty,1}} \& {} \& {\mathfrak{p}_{k+2,1}}
	\arrow[no head, from=1-1, to=3-1]
	\arrow[between={0}{0.4}, no head, from=1-3, to=3-2]
	\arrow[no head, from=1-3, to=3-3]
	\arrow[no head, from=1-4, to=3-3]
\end{tikzcd}$ \\[8ex]

4.  $\supp(\M_{k})$ & $\begin{tikzcd}[ampersand replacement=\&,sep=small]
	\&\&\&\&\&\&\&\& {\mathfrak{p}_{k,0}} \& {} \& {\mathfrak{p}_{2,0}} \& {\mathfrak{p}_{1,0}} \\
	{} \& {} \& {} \& {} \& {} \& {} \& {} \& {} \&\& \dots \\
	\&\&\&\&\&\&\&\& {\mathfrak{p}_{k,1}} \& {} \& {\mathfrak{p}_{2,1}} \& {\mathfrak{p}_{1,1}}
	\arrow[no head, from=1-9, to=3-9]
	\arrow[between={0}{0.4}, no head, from=1-11, to=3-10]
	\arrow[no head, from=1-11, to=3-11]
	\arrow[no head, from=1-12, to=3-11]
	\arrow[no head, from=1-12, to=3-12]
	\arrow[between={0}{0.4}, no head, from=3-9, to=1-10]
\end{tikzcd}$\\[6ex]
        \bottomrule
    \end{tabular}
    \vspace{0.1cm}
    \caption{It is useful to record the support of some familiar objects in $\Spc(\DQMgm(F; \F_2))$.}
    \label{tab:support_diagrams}
\end{table}

\begin{proof}
It suffices to show that a specialisation closed subset is of the form $(1)$ or $(2)$ if and only if $Z=\supp(A)$ for some object $A \in \DQMgm(F; \F_2)$. We begin by showing that for any object $A$, $\supp(A)$ has the desired form. 

\textbf{Case 1: $\p_{\infty,0} \in \supp(A)$}. In this case, 
\[U(A) \subseteq \Spc(\DQMgm(F; \F_2))\setminus \{\p_{\infty,0}\}=U(\cone(\rho)) \cup U(\tilde{\M}_1) \cup \bigcup_{j \geq 2} U(\M_{j-1} \oplus \tilde{\M}_j).\]
Since $U(A)$ is quasi-compact, and each open subset in the right hand side contains finitely many points, then $U(A)=\supp(A)^c$ is finite, which proves that $\supp(A)$ satisfies (1).

\textbf{Case 2: $\p_{\infty,0} \not \in  \supp(A)$}. The assumption implies that $A \in \p_{\infty,0}= \langle \M_n \mid n \in \N \rangle$.
Since $A$ is compact, then it is in fact built from finitely many Rost motives, i.e.
\[A \in \langle \M_{t_1}, \dots, \M_{t_s} \rangle \subseteq \langle \M_n \rangle,\]
for some sufficiently large $n$. Hence $\supp(A) \subseteq \supp(\M_n)$, and since $\supp(\M_n)$ is finite (cf. Table \ref{tab:support_diagrams}), we deduce that $\supp(A)$ satisfies (2).

This proves one direction, now we must prove that any specialisation-closed subset $Z$ of the form $(1)$ or $(2)$ can be written as the support of some object.

\textbf{Case 1: $\p_{\infty,0} \in Z$ and $Z$ is cofinite}. First, if $\p_{\infty,2} \in Z$, then $Z = \Spc(\DQMgm(F; \F_2)) = \supp(\T)$. If $\p_{\infty,2} \not \in Z$, then it must be contained in 
\[\supp(\cone(\rho) \oplus \cone(\tau))=\{\p_{\infty,2}\}^c .\]
By assumption, the (finite) complement $Z^c$ contains at most finitely many closed points $\p_{k,0}$. Because $Z$ is specialisation-closed, missing $\p_{k,0}$ implies missing $\p_{k,1}$ and $\p_{k+1,1}$; in particular, $Z$ is contained in a finite intersection of subsets of the form
\[\supp(\tilde{\M}_1)=\{\p_{1,0}, \p_{1,1}, \p_{\infty,2}\}^c, \quad \supp(\M_{k-1} \oplus \tilde{\M}_k)=\{\p_{k,0}, \p_{k,1}, \p_{k+1,1}, \p_{\infty,2}\}^c \, \, \text{ for } k \geq 2.\]
This means that $Z \subseteq \supp(Y)$ for some object $Y$ obtained as a tensor product of objects involved in the above finite intersection. If the inclusion is an equality, we are done; otherwise, we observe that the only points in $\supp(Y)$ that are not in $Z$ can only be of the form $\p_{k,1}$ for $k \in \N \cup \{ \infty \}$, and there can only be finitely many of these. Those can be removed by intersecting $\supp(Y)$ with finitely many subsets of the form
\[\supp(\cone(\rho))=\{\p_{\infty, 1}, \p_{\infty,2 }\}^c, \quad \supp(\cone(\tau))=\{\p_{1,1}, \p_{\infty, 2} \}^c, \quad \supp(\M_k \oplus \tilde{\M}_k)= \{\p_{k+1,1}, \p_{\infty,2}\}^c. \]
The resulting subset will be by construction equal to $Z$ and it can be written as the support of $Y$ tensored by objects involved in the above finite intersection.

\textbf{Case 2: $\p_{\infty,0} \not \in Z$ and $Z$ is finite}. Because it is specialisation-closed, $Z$ cannot contain $\p_{\infty,1}$ or $\p_{\infty,2}$. Since $Z$ is finite, $Z \subseteq \supp(\M_k)$ for a sufficiently large $k$. Any finite specialisation-closed subset of $\supp(\M_k)$ can be constructed as a finite union of the following:
\begin{itemize}
    \item  $\{\p_{k,0}\}=\supp(\tilde{\M}_{k-1} \otimes \M_k)$ for $ 1 \leq k < \infty$, and $\tilde{\M}_0:=\cone(\tau)$.
    \item $\overline{\{\p_{1,1}\}}=\supp(\M_1)$;
    \item $\overline{\{\p_{k,1}\}}=\supp(\tilde{\M}_{k-2} \otimes \M_k)$ for $2 \leq k < \infty$, and $\tilde{\M}_0:=\cone(\tau)$. 
\end{itemize}
From this it is clear that $Z$ can be expressed as the support of the direct sum of objects involved in the above finite union.
\end{proof}
From the description of the basis in the above theorem, it is straightforward to characterise all the closed subsets of the spectrum.
\begin{cor} \label{cor: closed sets of the spectrum}
\emph{The closed subsets of $\Spc(\DQMgm(F; \F_2))$ are specialisation closed subsets such that if $\p_{\infty,0} \not \in Z$, then $Z$ is finite.}
\end{cor}

\begin{remark}
We observe that $\Spc(\DQMgm(F; \F_2))$ is a countably infinite space of Krull dimension $2$. Moreover, from the description of the basis of closed subsets it is clear that the closed point $\p_{\infty,0}$ is not visible, i.e. it cannot be written as the support of some object. In particular, this implies by \cite{balmer2011generalized} that $\Spc(\DQMgm(F; \F_2))$ is not Noetherian. Of course, the latter fact may also be seen directly from the explicit description of the closed sets.
\end{remark}
\section{Applications}
\subsection{Restriction to Artin-Tate motives}
Under the inclusion 
\[\DATMgm(F; \F_2) \hookrightarrow \DQMgm(F; \F_2)\] 
we may inspect the generators of the prime ideals of $\DATMgm(F; \F_2)$ found in \cite[Theorem 10.1, Remark 10.16]{balmer2022three} to understand how points get mapped under the induced restriction
\[\Spc(\DQMgm(F; \F_2)) \twoheadrightarrow \Spc(\DATMgm(F; \F_2)).\]
The resulting map is summarised in the diagram below, where the colours indicate how points of $\Spc(\DQMgm(F; \F_2))$ get mapped to $\Spc(\DATMgm(F; \F_2))$.

\[\begin{tikzcd}[ampersand replacement=\&,sep=tiny]
	{\mathfrak{p}_{\infty,0}} \&\& {\mathfrak{p}_{k+1,0}} \& {\mathfrak{p}_{k,0}} \& {} \& {\mathfrak{p}_{3,0}} \& {\mathfrak{p}_{2,0}} \& {\textcolor{red}{\mathfrak{p}_{1,0}}} \\
	\& \dots \&\&\& \dots \\
	\& {} \& {\mathfrak{p}_{k+1,1}} \& {\mathfrak{p}_{k,1}} \& {} \& {\mathfrak{p}_{3,1}} \& {\textcolor{orange}{\mathfrak{p}_{2,1}}} \& {\textcolor{blue}{\mathfrak{p}_{1,1}}} \\
	{\textcolor{violet}{\mathfrak{p}_{\infty,1}}} \\
	\&\&\&\& {\textcolor{teal}{\mathfrak{p}_{\infty,2}}} \\
	\&\& {\mathcal{L}_1} \&\&\&\& {\textcolor{red}{\mathcal{N}_1}} \\
	\&\&\&\& {\textcolor{orange}{\mathcal{M}_1}} \\
	\&\& {\textcolor{violet}{\mathcal{L}_0}} \&\&\&\& {\textcolor{blue}{\mathcal{N}_0}} \\
	\&\&\&\& {\textcolor{teal}{\mathcal{M}_0}}
	\arrow[no head, from=1-1, to=4-1]
	\arrow[between={0}{0.4}, no head, from=1-3, to=3-2]
	\arrow[no head, from=1-3, to=3-3]
	\arrow[no head, from=1-4, to=3-3]
	\arrow[no head, from=1-4, to=3-4]
	\arrow[between={0}{0.4}, no head, from=1-6, to=3-5]
	\arrow[no head, from=1-6, to=3-6]
	\arrow[no head, from=1-7, to=3-6]
	\arrow[no head, from=1-7, to=3-7]
	\arrow[no head, from=1-8, to=3-7]
	\arrow[no head, from=1-8, to=3-8]
	\arrow["\dots"{description}, no head, from=3-2, to=5-5]
	\arrow[no head, from=3-3, to=5-5]
	\arrow[between={0}{0.4}, no head, from=3-4, to=1-5]
	\arrow[no head, from=3-4, to=5-5]
	\arrow[no head, from=3-6, to=5-5]
	\arrow[no head, from=3-7, to=5-5]
	\arrow[shift right=2, no head, from=4-1, to=5-5]
	\arrow["\dots"{description, pos=0.3}, between={0}{0.6}, no head, from=5-5, to=2-5]
	\arrow[shift right=2, no head, from=5-5, to=3-8]
	\arrow[between={0.2}{0.8}, maps to, from=5-5, to=7-5]
	\arrow[no head, from=6-3, to=7-5]
	\arrow[no head, from=6-3, to=8-3]
	\arrow[no head, from=6-7, to=7-5]
	\arrow[no head, from=6-7, to=8-7]
	\arrow[no head, from=7-5, to=9-5]
	\arrow[no head, from=8-3, to=9-5]
	\arrow[no head, from=9-5, to=8-7]
\end{tikzcd}\]
\begin{remark}
From the description of the generators of the prime tt-ideals we see that $\p_{2,1}$ and $\p_{\infty,2}$ are the only \emph{periodic} points of the spectrum in the sense of \cite{gallauer2025periods}, and both have period $1$. All remaining points have period $0$. A notable feature of the restriction map depicted above is that all the additional points of $\Spc(\DQMgm(F; \F_2))$ that are not seen by $\Spc(\DATMgm(F; \F_2))$ come from the fiber over $\mathcal{L}_1\in \Spc(\DATMgm(F; \F_2))$. This behaviour can be explained by considering a more general notion of periodicity for prime tt-ideals, which we plan to develop in future joint work with Sparks, where we will also return to the present example.
\end{remark}

\subsection{Restriction of Vishik's isotropic points}
Recall that in \cite{vishik2024isotropic,vishik2025balmer} Vishik constructed $2^{\mathfrak{c}}$ closed points of $\Spc(\DMgm(\R; \F_2))$, denoted $a_{2,E}$ and indexed by $2$-equivalence classes of field extensions $E/\R$. We study their restriction along the inclusion
\[
\DQMgm(\R; \F_2)\hookrightarrow \DMgm(\R; \F_2).
\]
There are several ways to see that  $a_{2,E}$ must restrict to a closed point of $\Spc(\DQMgm(\R; \F_2))$. One convenient criterion uses an important property of isotropic points: if an isotropic point contains a Rost motive $\M_n$ (resp.\ a reduced Rost motive $\tilde{\M}_n$), then it cannot contain $\tilde{\M}_n$ (resp.\ $\M_n$). This is a consequence of the stronger fact that isotropic points are of boundary type ( \cite[Example 3.13]{vishik2025balmer}). Moreover, $\cone(\rho)$ and $\cone(\tau)$ are not contained in any isotropic points, since $\rho$ and $\tau$ become zero under Vishik's isotropic realisation.  By inspection on the generators of the prime tt-ideals of $\Spc(\DQMgm(\R; \F_2))$, this forces
\[a_{2, E} \cap \DQMgm(\R; \F_2)=\p_{k,0}, \quad \text{for some } k \in \N \cup \{\infty\}.\]
We now explain how to determine $k$ from $E$ using Vishik's $2$-partial ordering on field extensions. For a field extension $E/\R$, we shall consider
\[
S_E:=\{\,n\in \N \mid E \twodom \R(Q_n)\,\}.
\]
Since
\[\R(Q_1)\stackrel{2}{>} \R(Q_2)\stackrel{2}{>} \cdots \stackrel{2}{>} \R(Q_n)\stackrel{2}{>} \cdots,\]
it is clear that $S_E$ is upward closed: i.e. if  $m \in S_E$ and  $m'\geq m$, then  $m' \in S_E$. We define
\[k(E):=
\begin{cases}
\infty & \text{if } S_E=\emptyset,\\
\min(S_E) & \text{if } S_E\neq \emptyset.
\end{cases}
\]

\begin{prop}
For any field extension \(E/\R\), we have that
\[
a_{2,E}\cap \DQMgm(\R; \F_2)=\p_{k(E),0}.
\]
\end{prop}

\begin{proof}
As explained above, $a_{2,E}$ must restrict to some $\p_{k,0}$. We determine $k$ by inspecting which Rost motives become $2$-isotropic after base-change to $E$. 

Suppose $S_E=\emptyset$, i.e. $E\not\twodom \R(Q_n)$ for all $n\in \N$. Then $(\M_n)_E$ is $2$-anisotropic for all $n$, hence $\M_n\in a_{2,E}$ for all $n$. This forces $a_{2, E} \cap \DQMgm(\R; \F_2)=\p_{\infty,0}$.

Assume now that $S_E\neq\emptyset$, and let $k=\min(S_E)$. Then $E\twodom \R(Q_k)$, so $(\M_k)_E$ is $2$-isotropic and therefore $\tilde{\M}_k\in a_{2,E}$. If $k=1$, this forces $a_{2,E}$ to restrict to $\p_{1,0}$. If $k \geq 2$, then we know that $E \not \twodom \R(Q_{k-1})$: in particular, $(\M_{k-1})_E$ is $2$-anisotropic, so $\M_{k-1} \in a_{2,E}$. From this, it is clear that $a_{2,E}$ must restrict to $\p_{k,0}$, and this concludes the proof. 
\end{proof}

\subsection{The integral spectrum of quadrics over $\R^{alg}$}
If the algebraic closure $\overline{F}$ of a real closed field $F$ satisfies an additional vanishing condition in motivic cohomology \cite[Hypothesis 6.6]{gallauer2019tt}--for instance, these are known to be satisfied by $F=\R^{\mathrm{alg}}$--then one can adapt the arguments of \cite[Theorem 11.3]{balmer2022three} to fully describe the Balmer spectrum of $\DQMgm(F; \Z)$. The first observation is that for any real closed field $F$ and coefficient ring $R$ in which $2$ is invertible, one has an identification
\[
\DQMgm(F; R)=\DATMgm(F; R).
\]
Indeed, if $2$ is a unit in $R$, then motives of smooth projective quadrics split into Tate motives and motives of quadratic extensions; over a real closed field these are precisely the generators of the category of Artin-Tate motives. In particular, for $R=\F_p$ with $p$ an odd prime, the tt-spectrum of $\DATMgm(F;\F_p)$, and hence of $\DQMgm(F;\F_p)$, is computed in \cite{balmer2022three}. For $R=\Q$, it is conjectured that $\Spc(\DATMgm(F;\Q))$ is a singleton, and known to be under \textit{Gallauer's vanishing condition} on $\overline{F}$.

With these inputs in place, we describe the integral spectrum of $\DQMgm(\R^{\mathrm{alg}}; \Z)$, and we conjecture that the same description holds for every real closed field.
\begin{cor}
\emph{Let $F=\R^{alg}=\overline{\Q} \cap \R$ be the field of real algebraic numbers, then $\Spc(\DQMgm(\R^{alg}; \Z))$ consists of the following points and specialisation relations
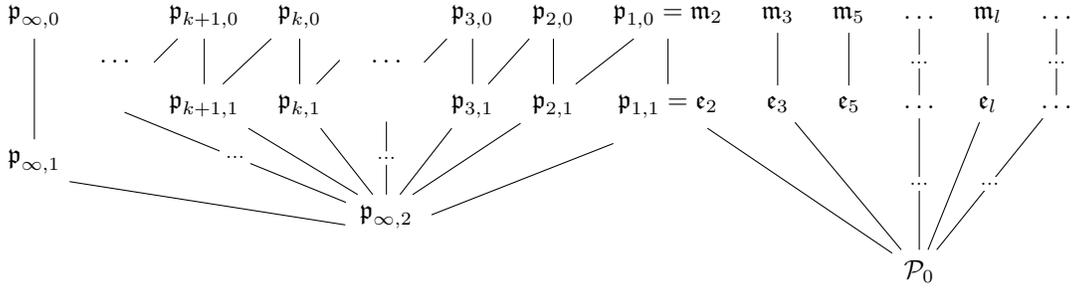
\[\begin{tikzcd}[ampersand replacement=\&,sep=tiny]
	{\p_{\infty,0}} \&\& {\p_{k+1,0}} \& {\p_{k,0}} \& {} \& {\p_{3,0}} \& {\p_{2,0}} \& {\p_{1,0}=\mathfrak{m}_2} \& {\mathfrak{m}_3} \& {\mathfrak{m}_5} \& \dots \& {\mathfrak{m}_l} \& \dots \\
	\& \dots \&\&\& \dots \\
	\& {} \& {\p_{k+1,1}} \& {\p_{k,1}} \& {} \& {\p_{3,1}} \& {\p_{2,1}} \& {\p_{1,1}=\mathfrak{e}_2} \& {\mathfrak{e}_3} \& {\mathfrak{e}_5} \& \dots \& {\mathfrak{e}_l} \& \dots \\
	{\p_{\infty,1}} \\
	\&\&\&\& {\p_{\infty,2}} \\
	\&\&\&\&\&\&\&\&\&\& {\mathcal{P}_0}
	\arrow[no head, from=1-1, to=4-1]
	\arrow[between={0}{0.4}, no head, from=1-3, to=3-2]
	\arrow[no head, from=1-3, to=3-3]
	\arrow[no head, from=1-4, to=3-3]
	\arrow[no head, from=1-4, to=3-4]
	\arrow[between={0}{0.4}, no head, from=1-6, to=3-5]
	\arrow[no head, from=1-6, to=3-6]
	\arrow[no head, from=1-7, to=3-6]
	\arrow[no head, from=1-7, to=3-7]
	\arrow[no head, from=1-8, to=3-7]
	\arrow[no head, from=1-8, to=3-8]
	\arrow[no head, from=1-9, to=3-9]
	\arrow[no head, from=1-10, to=3-10]
	\arrow["\dots"{description}, no head, from=1-11, to=3-11]
	\arrow[no head, from=1-12, to=3-12]
	\arrow["\dots"{description}, no head, from=1-13, to=3-13]
	\arrow["\dots"{description}, shift right, no head, from=3-2, to=5-5]
	\arrow[no head, from=3-3, to=5-5]
	\arrow[between={0}{0.4}, no head, from=3-4, to=1-5]
	\arrow[no head, from=3-4, to=5-5]
	\arrow[no head, from=3-6, to=5-5]
	\arrow[no head, from=3-7, to=5-5]
	\arrow[no head, from=3-8, to=6-11]
	\arrow[no head, from=3-9, to=6-11]
	\arrow["\dots"{description}, no head, from=3-11, to=6-11]
	\arrow[no head, from=3-12, to=6-11]
	\arrow["\dots"{description}, no head, from=3-13, to=6-11]
	\arrow[shift right=3, no head, from=4-1, to=5-5]
	\arrow["\dots"{description, pos=0.3}, between={0}{0.6}, no head, from=5-5, to=2-5]
	\arrow[shift right=3, no head, from=5-5, to=3-8]
\end{tikzcd}\]
where $l$ ranges over all primes. Moreover, the closed subsets are specialisation closed subsets $Z$ such that
\begin{itemize}
    \item if $\p_{\infty,0} \not \in Z$, then $Z \cap \overline{\{\p_{\infty, 2}\}}$ is finite;
    \item if $\mathcal{P}_{0} \not \in Z$, then $Z \cap \overline{\{\mathcal{P}_{0}\}}$ is finite.
\end{itemize}}
\end{cor}
\begin{proof}
The proof follows essentially by the same arguments of \cite[Theorem 11.3]{balmer2022three}. First, since $\M(\overline{F})$ has Euler characteristic 2, we may decompose
\[\Spc(\DQMgm(F; \Z))=\supp(\cone(2)) \cup \supp(\M(\overline{F})).\]
By \cite[Theorem 1.7]{balmer2018surjectivity} $\supp(\cone(2))$ identifies with the image of pull-back on spectra induced by
\[cc_2:\DQMgm(F; \Z) \to \DQMgm(F; \F_2).\]
Moreover $\Spc(cc_2)$ is injective: this may be seen from the generators for all the prime tt-ideals other than $\p_{1,1}$ and $\p_{\infty,2}$, but for these two it follows by the same reasoning in \cite[Theorem 11.3]{balmer2022three}. Moreover, $\supp(\M(\overline{F}))$ identifies with the image--again by \cite[Theorem 1.7]{balmer2018surjectivity}--of the injective pull-back map on spectra induced by base change 
\[p:\DQMgm(F; \Z) \to \DTMgm(\overline{F}; \Z),\]
where the injectivity follows from the same arguments in \cite[Theorem 11.3]{balmer2022three}. Moreover, the collisions of the images of the two maps follow from the same reasons in \textit{loc. cit.} 

The description of the topology may be obtained from the fact that $Z$ is closed if and only if 
\[Z_1:=Z \cap \supp(\cone(2)) \text{ and } Z_2:=Z \cap \supp(\M(\overline{F}))\] 
are closed. The rest follows by the description of the topology of the spectrum $\DQMgm(F; \F_2)$ in Corollary \ref{cor: closed sets of the spectrum}, and the spectrum of $\DTMgm(\overline{F}; \Z)$ computed in \cite{gallauer2019tt}, by taking the continuous bijective pull-back of these subsets under $\Spc(cc_2)$ and $\Spc(p)$.
\end{proof}

\printbibliography
\end{document}